\documentclass[11pt, a4paper, twoside]{amsart}
\usepackage{a4wide}

\usepackage[utf8]{inputenc}
\usepackage[english]{babel} 
\usepackage{amsthm}
\usepackage{amsmath}
\usepackage{mathtools}
\usepackage{amssymb}
\usepackage[all]{xy}
\usepackage{tikz}
\usepackage{enumerate} 
\usepackage{hyperref}
\usepackage{mathrsfs}
\usepackage{pgfplots}
\usepackage{tikz-3dplot}
\usetikzlibrary{arrows.meta,3d}
\usepackage{amsaddr}

\theoremstyle{plain}
\newtheorem{theorem}{Theorem}[section]
\theoremstyle{definition}
\newtheorem {definition}{Definition}[section]
\theoremstyle{remark}
\newtheorem{remark}{Remark}[section]
\theoremstyle{plain}

\theoremstyle{plain}
\newtheorem{proposition}{Proposition}[section]
\theoremstyle{plain}
\newtheorem{corollary}{Corollary}[section]
\theoremstyle{plain}

\theoremstyle{plain}
\newtheorem{lemma}{Lemma}[section]

\def\R{\mathbb{R}}
\def\S{\mathbb{S}}

\def\N{\mathbb{N}}
\def\curl{\text{curl}}
\def\dive{\text{div}}

\def\cyl{{W_{\epsilon,\delta}}}

\newcommand{\mps}[1]{M_3(\Psi^{#1}(\mathbb{R}^4))}
\newcommand{\mpsc}[1]{M_3(\Psi^{#1}_{cl}(\mathbb{R}^4))}

\begin{document}

\title[Recovering a Metric from Cherenkov Radiation]{Recovering a Riemannian Metric from Cherenkov Radiation  in Inhomogeneous Anisotropic Medium}
\author{Antti T. P. Kujanp\"a\"a}
\email{antti.kujanpaa@helsinki.fi}
\address{Department of Mathematics and Statistics, University of Helsinki}

\begin{abstract}
Although travelling faster than the speed of light  in vacuum is not physically allowed, 
 the analogous bound in medium can be exceeded by a moving particle.  
For an electron in dielectric material this leads to emission of photons which is usually referred to as Cherenkov radiation. 
In this article a related 
mathematical system for waves in  inhomogeneous anisotropic 
medium with a maximum of three polarisation directions is studied. 
The waves are assumed to satisfy 
$P^k_j u_k (x,t)  = S_j(x,t)$, 
where $P$ is a vector-valued wave operator that depends on a Riemannian metric and $S $ is a point source that moves at speed $\beta < c$ in given direction $\theta \in \S^2$. The phase velocity $v_{\text{phase}}$ is described by the metric and depends on both location and direction of motion. 
In regions where $v_{\text{phase}}(x,\theta)  < \beta <c $ holds the source generates a cone-shaped front of singularities that propagate according to the underlying geometry.  
We introduce a model for a measurement setup that applies the mechanism and show that the Riemannian metric inside a bounded region can be reconstructed from partial boundary measurements. 
The result suggests that Cherenkov type radiation can be applied to detect internal geometric properties of an inhomogeneous anisotropic target from a distance. 
\end{abstract}

\maketitle

%
%
%

\section{Introduction}

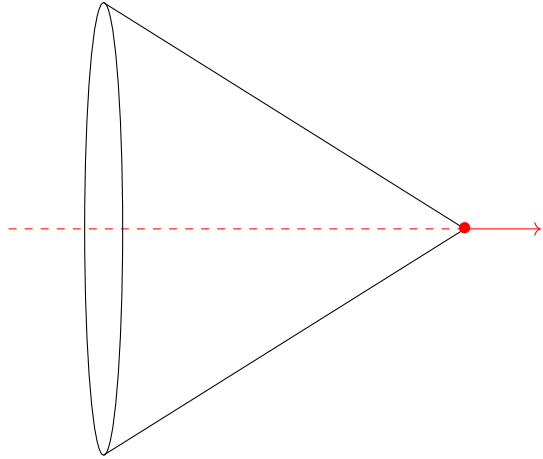
\begin{figure}[h]
\begin{tikzpicture}[scale=0.5]

 \definecolor{darkestco}{RGB}{0,0,0}
   \definecolor{darkco}{RGB}{0,80,80}
   \definecolor{co}{RGB}{0,160,195}
   \definecolor{lightco}{RGB}{0,180,225}
  \definecolor{lightestco}{RGB}{0,200,255}
    \definecolor{lightestlightestco}{RGB}{10,220,255}
     \definecolor{ultralightestlightestco}{RGB}{10,255,255}
  
 \draw[dashed,color=red] (-2,0) -- (10,0);
%

%
%
%
%
%
%
%
%
%
%
%

  \draw[color=darkestco] (0.5,6)--  (10,0)  --  (0.5,-6);

      \draw[color=darkestco] (0,0) arc (180:360:0.5cm and 6cm) ;
  \draw[color=darkestco] (0,0) arc (180:0:0.5cm and 6cm);
  
    
%
%
%
%
   \draw[color=red] (10,0) node {\textbullet};
 
 \draw[color=red][->] (10,0) -- (12,0);

%
%
%
\end{tikzpicture}

\caption{Cherenkov radiation wave front (in black) in three dimensional medium at a fixed time is a cone with apex at a moving charged particle (in red). 
The phenomenon is an electromagnetic equivalent of a sonic boom. 
 }\label{90ds}
\end{figure}

When a charge carrier moves in material faster than the phase velocity it radiates photons. This phenomenon, known in physics  as Cherenkov radiation, plays an important role in particle detection systems such as the ring imaging Cherenkov detector (RICH). 
The mechanism has mathematical description (See Section \ref{micromikko}) as a modification of propagation of singularities for real principal type operators.  To introduce the concept, let us consider a system 
of the form
\[
P^k_j u_k (x,t)  = S_j(x,t), \quad j=1,2,3, \quad x\in \R^3 , \quad t\in \R,
\]
where the Einstein summation convention is considered over $k=1,2,3$, the operator $P$ with entries $P^k_j = P^k_j (x,t,D)$, $j,k=1,2,3$ is hyperbolic,  
and $S_j (x,t)  dx^j$ is a singular source with each coordinate $S_j$ conormal to the world line 
\[
K:= \{ (z+t \beta \theta, t) \in \R^4 : t\in \R \}, \quad \theta \in \S^2, \quad  \beta\in (0,1)
\]
 (i.e. trajectory in space-time) of the particle. 
An example of this is a scalar model, e.g. the scalar wave equation 
\[
(\partial_t^2 - \Delta_g )u = f,\quad u  \in \mathcal{D}'(\R^4),
\]
(identified as $P^k_j = \delta^k_j  (\partial_t^2 - \Delta_g ) $ and $S_j=f$)
with a source $ f \in \mathcal{D}'(\R^4)$ conormal to $K$ (e.g. $f(x,t)=\delta(x-z-\beta \theta)$).  
The source $S=S_j dx^j$ can be interpreted as a singularity that moves at a physically allowed ($\beta<c=1$) constant velocity $\beta \theta$, $\beta \in (0,1)$ through $z\in \R^3$. 
If for $P$, or more generally, for the product $V P W$ with non-degenerate $V$ and $W$ the leading part of the operator equals a scalar operator $Q = QI$ of real principal type,  then
the standard propagation of singularities extends to the case above. 
More precisely; the tensors $V$ and $W$ are neglectable in terms of singularities and the wave front set of $u = u_j dx^j$ splits into a non-propagating part which does not leave $\bigcup_{j=1}^3 WF(S_j) \subset N^*K$, and a propagating part which spreads along the bicharacteristics of $Q$ into the surrounding space. 
We assume that the leading part $Q$ is of the form 
\[
Q= \partial_t^2 - g^{jk}(x)  \frac{\partial}{\partial x^j} \frac{\partial}{\partial x^k} 
\]
where $g$ stands for a Riemannian metric. The characteristic manifold for $Q$ is the covector light-cone 
\[
L\R^4 = \{ (x,t;\xi,\omega) \in T^*\R^4 \setminus \{0\}  : g^{jk}(x) \xi_j \xi_k = \omega^2 \}, 
\]
also known as the null-cone, and the speed of wave propagation at $x\in \R^3$ in direction $\theta \in \S^2$ is $ v_{\text{phase}}(x,\theta):=  ( g^{jk}(x) \theta_j \theta_k )^{-1/2}$. 
Provided that the source moves slower than waves
(i.e. $\beta < v_{\text{phase}}$) 
the normal bundle $N^*K$ of the world line, and hence the wave front set of the source remains disjoint from the light-cone (the left picture in Figure \ref{asd34iuw2222}) which implies that no propagating singularities are generated. This corresponds to absence of Cherenkov radiation in subluminal regimes. 
On the other hand, in regions where the point source moves faster than waves 
a non-empty intersection between the light-cone and the normal bundle appears (the right picture in Figure \ref{asd34iuw2222}), thus enabling radiation of singularities from the source.  
The propagating singularities add up to a conical shock wave which is illustrated in Figure \ref{90ds}. 

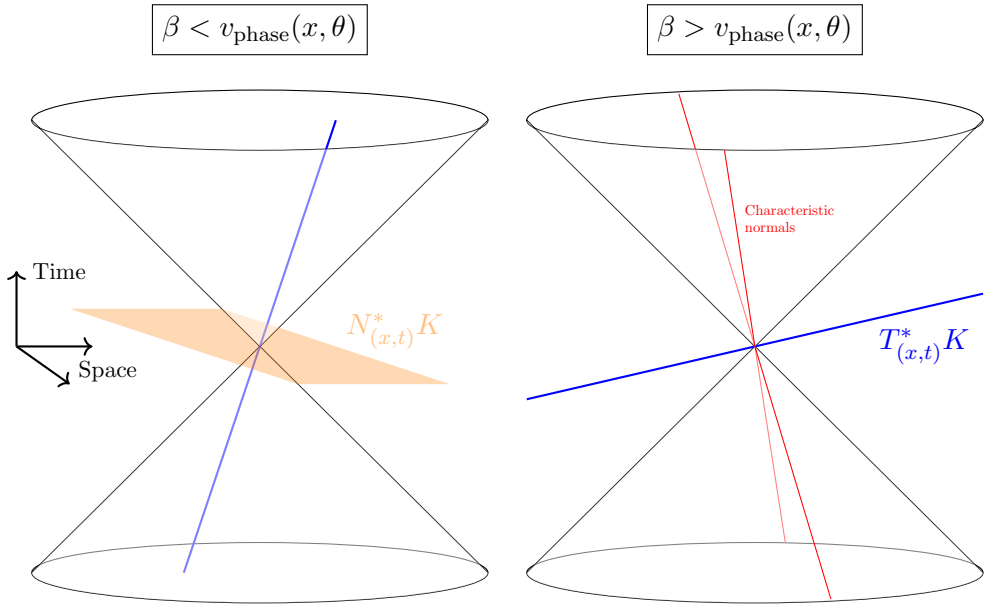
\begin{figure}[h]
\begin{minipage}{.43\textwidth}
\begin{tikzpicture}

\node[rectangle,draw] at (0,4.2) {
$\beta < v_{\text{phase}}(x,\theta) $
} ; 

    \draw (-3,-3) arc (180:360:3 and 0.4) -- (0,0) -- cycle;
  \draw (-3,-3) arc (180:0:3 and 0.4);
 
%
   
  \draw (-3,3) arc (180:360:3 and 0.4) -- (0,0) -- cycle;
  \draw (-3,3) arc (180:0:3 and 0.4);

%
%

  \draw[thick,color=blue] (-1,-3) -- (1,3);

    \fill[white,opacity=0.5] (-3,-3) arc (180:360:3 and 0.4) -- (0,0) -- cycle;
      \fill[color=orange,opacity=0.3] (-2.5,0.5) -- (0.5,-0.5) -- (2.5,-0.5) --(-0.5,0.5) ;
         \fill[white,opacity=0.5] (-3,3) arc (180:360:3 and 0.4) -- (0,0) -- cycle;
         
         \node[right,color=orange,opacity=0.5] at (1,0.2)  {$N^*_{(x,t)} K$};

           \draw[->, thick] (-3.2,0) -- (-3.2,1);
      \node[right,scale=0.8] at (-3.1,1)  {Time};
      
         \draw[->, thick] (-3.2,0) -- (-2.2,0);
           \draw[->, thick] (-3.2,0) -- (-2.5,-0.5);
         \node[below,scale=0.8] at (-2,-0.08)  {Space};
  
  \end{tikzpicture}
\end{minipage}
  \begin{minipage}{.43\textwidth}

\begin{tikzpicture}

\node[rectangle,draw] at (0,4.2) {
$\beta > v_{\text{phase}} (x,\theta)$ 
} ; 

    \draw (-3,-3) arc (180:360:3 and 0.4) -- (0,0) -- cycle;
  \draw (-3,-3) arc (180:0:3 and 0.4);
%
   
  \draw (-3,3) arc (180:360:3 and 0.4) -- (0,0) -- cycle;
  \draw (-3,3) arc (180:0:3 and 0.4);

%
%

   \draw[color=red] ( 0,0) -- ( -1,3.35) ; 
    \draw[color=red] ( 0,0) -- ( 0.4,-2.6) ; 
  
  \draw[thick,color=blue] (-3,-0.7) -- (3,0.7);

     \fill[white,opacity=0.5] (-3,3) arc (180:360:3 and 0.4) -- (0,0) -- cycle;
   \fill[white,opacity=0.5] (-3,-3) arc (180:360:3 and 0.4) -- (0,0) -- cycle;
   
     \draw[color=red] ( 0,0) -- ( -0.4,2.6) ; 
      \draw[color=red] ( 0,0) -- ( 1,-3.35) ; 
      
      
       \node[right,color=blue] at (1.5,0)  {$T_{(x,t)}^*K$};
       
        \node[right,color=red,scale=0.5,text width=3cm] at (-0.2,1.7)  {Characteristic normals};
      
\end{tikzpicture}


\end{minipage}
\caption{The set of characteristic covectors of the wave operator in a single fibre $T_{(x,t)}^*\R^4$ is the light-cone $L_{(x,t)}\R^4$ which in some convenient spatial coordinates (e.g. normal coordinates at $x$) is the Minkowsky light-cone (in black). 
The blue lines illustrate the cotangent space of the world-line $K$ through $(x,t) \in K$. The two pictures correspond to subluminal and superluminal cases, respectively. That is; the line $K$ is time-like at $(x,t)$ in the first picture, whereas at the second picture the source moves faster than null-geodesics. The normal plane $N_{(x,t)}^*K$ (in orange) of the line at $(x,t)$ does not intersect the cone  in the first picture, whereas 
in the second picture there exist characteristic normals (in red), and hence bicharacteristics in $L\R^4$ through them.}\label{asd34iuw2222}
\end{figure}

In this article it is shown that Cherenkov radiation can be applied to recover Riemannian structure of a target region within anisotropic inhomogeneous medium.
The idea is to send singularities (particles or sharp pulses) into the region of interest and observe new singularities generated as the incident signals interact with the medium. These disruptions are carried by waves into the exterior of the target where they eventually hit a given hypersurface that surrounds the region of interest. 
On the surface the generated singularities are observed and data is collected by repeating the measurement with various initial parameters $(z,\theta,\beta)$ of the source. We show that information collected in this way on a convenient part of the surface determines uniquely the metric $g$ inside the target for a large class of geometries and sources.  

The structure of the article is as follows: In the first section the model and the main result together with two related examples are introduced. Some basic concepts are briefly discussed in Section \ref{premmi}. In Section \ref{micromikko} we build up the microlocal tools which are then used in the last section for proving the main result. 
Appendix \ref{appe1} contains additional details for Section \ref{micromikko}.

\subsection{The Model}\label{modella}

Let $I$ stand for the identity matrix. 
Assume that the (1,1)-tensor $P(x,t,D)$ on $\R^3$ with differential operator entries $P^k_j(x,t,D) \in \Psi^2_{cl} (\R^4)$ is similar to the vector wave-operator 
in the sense that the condition
\begin{equation}\label{sdfawh}
P(x,t,D) = \left( \partial_t^2 - g^{rl}(x)  \frac{\partial}{\partial x^r}  \frac{\partial}{\partial x^l} \right) I + F(x,D_x) ,
\end{equation}
that is, 
\begin{equation}\label{eroeoeeo}
  P^k_j    (x,t,D)    =  \delta_j^k  \left(\partial_t^2 - g^{rl}(x) \frac{\partial}{\partial x^r} \frac{\partial}{\partial x^l} \right)    + F^k_j (x,D_x) , \quad j,k=1,2,3, 
\end{equation}
holds for some first order operators $F_j^k(x,D_x)$, $j,k=1,2,3$. 
%
Consider a point-like object moving linearly in $\R^3$ in direction $\theta  \in \S^2 $ at physically allowed speed $\beta \in (0,c)=(0,1)$ ($c=1$ in natural unit system). Let $z$ be the location at time $t=0$. The associated world line in the space time $\R^4$ is 
\[
K(z,\theta,\beta) := \{ (x,t) \in \R^4 : x= z + t \beta \theta \}, 
\]
We study singularities of distributions $v_j\in \mathcal{D}' ( \R^4)$, $j=1,2,3$, that obey 
\begin{equation}\label{eiSD}
P^k_j v_k  (x,t)       \equiv  S_j (x,t) \mod C^\infty( \R^4)   , \quad j=1,2,3, 
\end{equation}
where 
 the terms $S_j$ are conormal distributions over $K(z,\theta,\beta)$, that is,
 \begin{equation}\label{woeser3}
 S_j (x,t)  = \int_{\R^3} e^{i(x-z-t\beta \theta) \cdot \xi  } c_j(x,t,\xi) d\xi  ,  \quad j=1,2,3, \quad c_j \in S^r ( \R^4 \times \R^3).
 \end{equation}
 E.g. $S= \delta (x- \beta \theta t )  dx^1 $ if $c_j = \delta_{j1}  $. 
Above the relation $a \equiv b \mod C^\infty $ refers to $a-b \in C^\infty$.  
The source $S=S_j dx^j$ can be interpreted as a  point singularity  that moves in $\R^3$ along the trajectory $t \mapsto z + t \beta \theta$, $\theta \in S^2$, $\beta\in (0,1)$, whereas the field $v= v_j dx^j$ is a wave in anisotropic inhomogeneous material. 
The three components of each element correspond to different polarisations. 
Electromagnetic waves created by a moving charge disruption such as a charged particle in anisotropic medium is perhaps the most obvious application for the model. 
For materials with scalar wave impedance the electric field obeys \eqref{eiSD} 
for a metric $g$ conformally equivalent to the anisotropic permeability tensor $\varepsilon(x)$. This model is explored in Section \ref{eleex}.
\begin{remark}\label{rekka1}
Notice that an analogous scalar equation, e.g the scalar wave equation 
\[
(\partial_t - \Delta_g)   v =f \in \mathcal{D}'(\R^4), \quad f,v \in \mathcal{D}'(\R^4)
\]
 with source $f$ conormal to $K(x,\theta,\beta)$, can also be expressed within the model by identifying any scalar operator $P_{scal}$ with the diagonal matrix  $ P^k_j :=  \delta^k_j  P_{scal}  $, $j,k=1,2,3$ and a scalar source $S_{scal}$ with $S_j :=S_{scal}$, $j=1,2,3$. 
Similarly, the model also generalises waves with two polarisation directions. 
\end{remark}
\begin{remark}\label{aksdhwri}
The model above is not as general as possible. 
It most likely suffices to assume that each entry $F^k_j$ of the tensor $F$ is an operator in the space $\Psi^1_{cl}(\R^3)$  of classical first order pseudo-differential operators. 
Moreover, 
the 
identity \eqref{sdfawh} 
can be replaced by the more general condition
\begin{equation}\label{wqie}
P(x,t,D) \cong \left( \partial_t^2 - g^{rl}(x)  \frac{\partial}{\partial x^r}  \frac{\partial}{\partial x^l} \right) I + F(x,D_x)
\end{equation}
where the relation $P \cong G$ between two operators $P$ and $G$ refers to 
\[
P^k_j = A^l_j G_l^r B_r^k
\]
for some operators $A$ 
 and $B$ 
 with pointwise non-degenerate principal symbol matrices $[\sigma_{m_1}(A^l_j )]$ and $[\sigma_{m_2}(B^l_j )]$ (see Lemma \ref{wiikki}). Notice that the non-degeneracy is just a vector-valued equivalent of ellipticity and that 
   a non-degenerate smooth  time-dependent $(1,1)$-tensor $A: u(x,t) \mapsto A^k_j(x,t) u_k(x,t) dx^j$ generalizes the concept of elliptic multiplication operator $u(x,t) \mapsto f(x,t)u(x,t)$ for non-vanishing $f\in C^\infty(\R^4)$. Notice that in the vector valued setting ellipticity allows exchange of singularities between polarisations. 
   
Moreover, the results of this work are expected to remain true even if the model is considered only microlocally near the covector light-cone 
 \[
 L\R^4 =  \{ (x,t,\xi,\omega)  \in T^*\R^4 : \omega^2 = g^{jk}(x) \xi_j \xi_k \}. 
 \]
 This follows from the fact that methodologically the proofs are microlocal and based on propagation of singularities. 
\end{remark}

 \subsubsection{The Data}\label{daatta}

Due to physical background we require that waves propagate with speed less or equal to the speed of light in vacuum. This is equivalent to 
\begin{equation}\label{physsi}
|\theta|_g(x) \geq 1, \quad  \forall (x,\theta) \in \R^3 \times \S^2, 
\end{equation}
where we denote $|\theta|_g(x) := \sqrt{\theta^j \theta^k g_{jk}(x)}$.
Let $W \subset \R^3$ be an open bounded set with smooth boundary $\partial W$ and fix open $U \subset W$.  
We consider $U$ as a region with unknown geometry inside the medium, whereas the boundary $\partial W$ is the hypersurface on which observations are made.
The source $S$ is controlled by the variables $z,\theta$ and $\beta$.

To avoid Cherenkov radiation from arbitrarily far we consider mediums that converge to vacuum at infinity. That is;
$|\theta|_g(x) \longrightarrow 1$ for every $\theta \in \S^2$ as $\sqrt{x\cdot x} \longrightarrow \infty$.
Alternatively, one may assume that there is a relatively large velocity threshold $\beta_0 \in (0,1)$ such that the preimage
\begin{equation}\label{pwe}
R_\theta = \Big\{ x \in \R^3 :  \frac{1}{\beta_0} <  | \theta |_g (x)  \Big\}
\end{equation}
is bounded and contains $U$ for every $\theta \in \S^2$, and then focus on velocities $\beta \in (0,\beta_0)$. 
In both cases the wave front set of the tail $(1-\chi)S$ is disjoint from the characteristic set of $P$ for a test function 
$\chi = \chi_{g,\beta} \in C_c^\infty(\R^3)$, $0 \leq \chi \leq 1$ that equals $1$ in a sufficiently large open set. Hence, for a convenient operator $P^{-1}$ (see Lemma \ref{wiikki}) that inverts $P$ microlocally near $N^*K $ in $ \{ \chi < 1\}$ the element $P^{-1}(1-\chi)S$ does not radiate singularities and 
 $u := v- P^{-1}(1-\chi)S$ satisfies
\begin{equation}\label{cytor}
Pu \equiv \chi S\in \mathcal{E}'(\R^4). 
\end{equation}
The advantage in writing the system in this form is that the source on the right is compactly supported. 
%
Through this procedure we may pass on 
to a setup 
\[
P^k_j u_k \equiv S_j,
\]
 in which the source $S = S_j dx^j$ is not only conormal to $K(z,\theta,\beta)$ but also compactly supported. 
\begin{definition}\label{stabbel}
Let $U\subset W \subset (\R^3,g)$ be as above. 
We say that a subset $\Upsilon = \Upsilon_{U,g} \subset \partial W$ is a stable part of the boundary with respect to $U$ if $\Upsilon$ is open in $\partial W$ and for every $z\in U$ there is at least one $x_0=x_0(z)\in \Upsilon$ that minimises distance to $\partial W$, i.e.
\[
\text{dist}_g(z, \partial W) = \text{dist}_g(z,x_0),
\]
where $\text{dist}_g(z, \partial W) := \min_{x \in \partial W} \text{dist}_g(x,z)$. 
\end{definition}
\begin{remark}
The whole boundary $\Upsilon= \partial W$ is always stable. 
\end{remark}

For $\Lambda \subset  T^*\R^4$ we let  $\Lambda|_{T(\Upsilon \times \R)} \subset  T^*(\Upsilon \times \R)$ stand for the canonical projection $\Lambda|_{T(\Upsilon \times \R)} := \Big\{ \big(x,t, (\xi,\omega)|_{T_{(x,t)} (\Upsilon \times \R)} \big)   : (x,t,\xi,\omega) \in \Lambda, (x,t) \in \Upsilon \times \R  \Big\}$. 
Given a family $ \{  S_j( \ \cdot \ ; z,\theta,\beta)dx^j :  (z,\theta,\beta) \in U \times \S^2 \times (0,1) \}$ of compactly supported sources of the form \eqref{woeser3} and a stable $\Upsilon\subset \partial W$, the measurement data are defined as the following map:
\[
 (z,\theta,\beta) \mapsto \bigcup_{j=1,2,3} WF( u_j  )|_{ T(\Upsilon \times \R) },  \quad (z,\theta,\beta) \in U \times \S^2 \times (0,1),
\]
where $u = u_j ( \ \cdot \ ; z,\theta,\beta) dx^j$ solves for each $ j=1,2,3,$ the conditions 
\begin{align} 
&P^k_j u_k \equiv S_j \mod C^\infty ( \R^4), \label{qwerijasdasdhhk111} \\
&u_j |_{\R^3 \times (-\infty, -T]} \in C^\infty( \R^3 \times ( -\infty , -T]  )  \label{qweri423423jhhk1} 
\end{align}
 for $S_j = S_j ( \ \cdot \ ; z,\theta,\beta) $ and some (sufficiently large) $T>0$. 
Notice that each source is supported in $\R^3 \times (-T,\infty)$ for large $T>0$ (see Figure \ref{efigg34356}). As the high-energy limit $\beta \longrightarrow 1$ is unpractical, we are interested in smaller datasets where the velocities are restricted to a smaller subinterval $\mathcal{I}_z \subset (0,1)$. The interval may depend on $z\in U$. 

\begin{remark}
Substituting $P$ in the definition with more general operator $\tilde{P} \cong P$ 
(see Remark \ref{aksdhwri}) yields the same data. 
For example, the data are invariant in transformation $u(x,t) \mapsto V^k_j(x,t) u_k(x,t) dx^j$ and $S(x,t) \mapsto W^k_j(x,t) S_k(x,t) dx^j$, where $V$ and $W$ are smooth non-degenerate $(1,1)$-tensors. 
\end{remark}

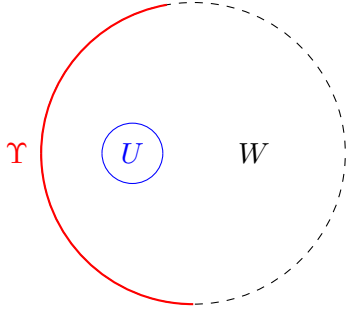
\begin{figure}[h]
\begin{tikzpicture}[scale=0.4]
 \draw[dashed] (0,0) arc (180:360:5cm and 5cm) ;
  \draw[dashed] (0,0) arc (180:0:5cm and 5cm);
   \draw[color=blue] (2,0) arc (180:360:1cm and 1cm) ;
  \draw[color=blue] (2,0) arc (180:0:1cm and 1cm);
  
   \draw[thick,color=red] (0,0) arc (-180:-90:5cm and 5cm);
     \draw[thick,color=red] (0,0) arc (180:100:5cm and 5cm);
     
     \node[color=blue] at (3,0) {$U$};
     
     \node at (7,0) {$W$};
     
      \node[color=red] at (-0.8,0) {$\Upsilon$};
      
     
  \end{tikzpicture}
  \caption{A schematic 2D illustration of $W$,$U$ and a stable part $\Upsilon \subset \partial W$. The ambient space represents $\R^3$.}
\end{figure}

\begin{figure}
\begin{tikzpicture}[scale=1.2]
    \draw[dashed] (0,-3) arc (180:360:1 and 0.1)  ;
  \draw[dashed] (0,-3) arc (180:0:1 and 0.1);
      \draw[color=blue] (0.4,-3) arc (180:360:0.5 and 0.05)  ;
  \draw[color=blue] (0.4,-3) arc (180:0:0.5 and 0.05);
    \draw[thick,color=red] (0,-3) -- (0,3) ; 
  \draw[dashed] (2,-3) -- (2,3) ; 
  \draw (-1,-2) -- (3,2);
      \draw[color=blue] (0.4,-3) -- (0.4,3) ; 
  \draw[color=blue] (1.4,-3) -- (1.4,3) ; 
  
  \node[label=above:z] at (1,0) {\textbullet};
   
   \draw[dotted] (-3,0) -- (5,0);
     \node[label=above: {$t=0$}] at (-1.8,0) { };
      \node[label=above: {$t \leq -T$}] at (-1.5,-4.5) { }; 
      \fill[color=gray, opacity=0.2] (-3,-3) -- (5,-3) -- (5,-5) -- (-3,-5);

         \draw[color=red, thick] (0,-3) arc (-180:-50:1 and 0.1)  ;
             \draw[color=red] (0,-3) arc (180:120:1 and 0.1)  ;
         
             \draw[color=red] (0.49,-2.92) -- (0.49,3) ; 
  \draw[color=red,thick] (1.65,-3.08) -- (1.65,3) ; 
      
\end{tikzpicture}
\caption{A schematic illustration of the sets associated with  (\ref{qwerijasdasdhhk111}-\ref{qweri423423jhhk1}). 
We consider waves $u$ that are smooth in the half-space $\R^3 \times (-\infty,-T]$,  indicated in gray. The sets $U \times (-T,\infty)$ and $\Upsilon \times (-T,\infty)$ are outlined in blue and red, respectively. Singularities of the source (the black line) lie in a compact subset  of $K(z,\theta,\beta) \cap (\R^3 \times (-T,\infty))$ for large $T>0$.}\label{efigg34356}
\end{figure}
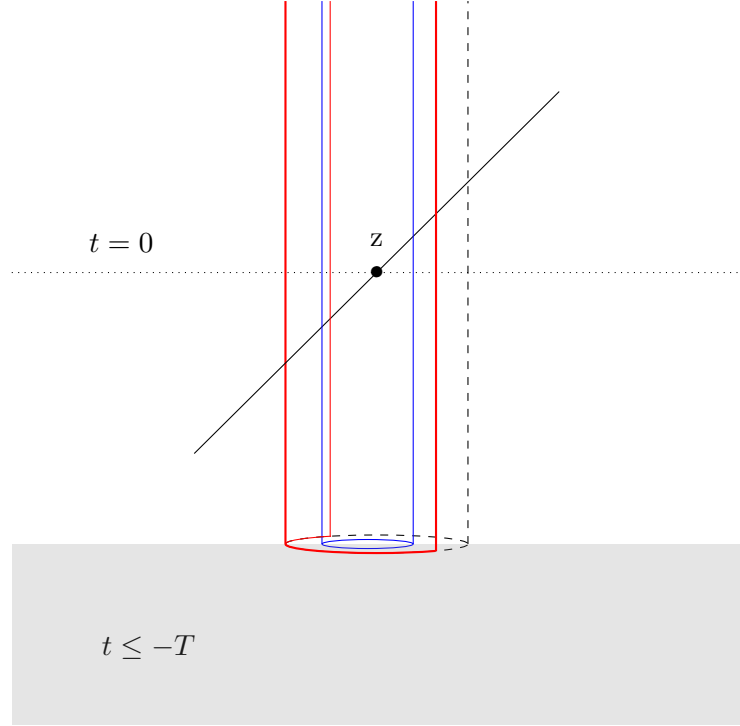

%
%

\subsection{The Main Result}

The main outcome of this work is Theorem \ref{maineee} below which states that the Riemannian metric $g$ in $U$ is uniquely determined by the data, provided that some natural conditions are satisfied. 
In fact, for each $z\in U$ it suffices to consider data for velocities $\beta$ in a specific subinterval $\mathcal{I}_z \subset  (0,1)$ with $\sup \mathcal{I}_z  $ strictly less than $c=1$, the speed of light in vacuum. 
If the maximum phase velocity in the closure  $\overline{U}$ is strictly less than the speed of light in vacuum then the dependence on $z\in U$ in the interval above can be dropped (see Remark \ref{labne}). 
The requirements for the metric are as follows:

\begin{definition}\label{admi}
Let $g$ be a Riemannian metric in $\R^3$ and let $U\subset  \R^3$ be an open bounded set.    
We say that the metric $g$ is admissible with respect to $U$ if the following conditions hold: 
\begin{enumerate}[(i)]
\item Propagation speed of waves is everywhere less than or equal to the speed of light in vacuum:  \[   |\theta|_g (x)  \geq 1,  \quad \forall (x,\theta) \in \R^3 \times \S^2. \]  \label{memma1}
\item Propagation speed of waves in $U$ is strictly less than the speed of light in vacuum:  
\[ |\theta|_g(z) > 1,  \quad \forall (z,\theta) \in U \times \S^2.  \label{memma2}\] 
\end{enumerate}
\end{definition} 
To avoid technical difficulties, the following lower bound for speeds $\beta$ is considered in the main theorem. 
  
\begin{definition}
Let $g$ be admissible metric with respect to $U\subset \R^3$. 
For  $z\in U$ and a bounded set $W \subset \R^3 $ containing $U$ we define 
\[
J_z := \Big( \max_{\theta \in \S^2} \frac{1}{|\theta|_g(z) } , 1\Big), \quad \beta_{\partial W} (z;g) :=  \inf \big\{ \beta \in  J_z  :  \beta  \text{dist}_g ( z, x) >  |x-z|_{eucl} , \ \forall x \in \partial W \big\} , 
\]
where $| \cdot |_{eucl} $ is the Euclidean norm. 
\end{definition}
\begin{remark}
It follows from compactness of $\partial W$ that   $\beta_{\partial W} (z;g) < 1$. Moreover, if propagation of waves in the closure $\overline{U}$ is strictly slower than in vacuum, we have $\sup_{z\in U} \beta_{\partial W} (z;g)<1$. 
\end{remark}

Below $M_3(X)$ refers to $3\times 3$ matrices with entries in the class $X$. In orthonormal Euclidean coordinates we identify a (1,1)-tensor $T^k_j  \frac{\partial}{\partial x^k} \otimes dx^j$, $T^k_j \in X$, with the matrix $[T^k_j] \in M_3( X) $.  
 %
%
%
The sources $S_j = S_j( \ \cdot \   ; z,\theta,\beta)$ are required to satisfy for every $ (z,\theta,\beta) \in U\times \S^2 \times (0,1)$, $j=1,2,3$ the following conditions:
\begin{equation}  \label{meq1} 
\begin{split}
& S_j  \in I_{cl}^{r+1/2} ( \R^4 ; N^*K(z ,\theta,\beta)),  \quad  r\in \R, \quad  \text{i.e.}  
\\ &    S_j(x,t) = \int_{\R^3} e^{i(x-z-t \beta \theta)\cdot \xi} a_j(x,t, \xi ) d\xi , \quad  a_j = a_j ( \ \cdot \ ; z,\theta,\beta) \in S^{r}_{cl} ( \R^4 \times \R^3),
\end{split}
\end{equation}
\begin{align}
&\text{singsupp} ( S_{j} ( \ \cdot \   ; z,\theta,\beta)) \subset B(0,R_\beta) \times \R, \quad \text{(cf. the cut-off in \eqref{cytor}) }  \label{meq2} \\
&  \text{$(a_{1,r} , a_{2,r} , a_{3,r}) (z,0,\xi)  \neq 0, \quad  \forall (\xi,\omega) \in  N^*K(z,\theta,\beta)  \cap L_{(z,0)} \R^4$ \quad  $( \omega = \pm |\xi|_g)$  }  \label{meq3}
\end{align}
where $a_{j,r}$ is the positively homogeneous principal part  of the symbol $a_j \sim a_{j,r} + a_{j,r-1} +a_{j,r-2} + \cdots$, $  L_{(z,0)} \R^4 :=  \{ (\xi,\omega) \in T^*_{(z,0)} \R^4 \setminus \{0\} : |\xi|_{g_\lambda}^2(z) = \omega^2\} $ and the radius $R_\beta>0$ depends continuously on $\beta$. 
%
The main result of the paper is the following theorem:
\begin{theorem} \label{maineee}
Let $W \subset \R^3$ be a bounded open set with smooth boundary $\partial W$ and 
 let $g_\lambda = g_{\lambda, jk}(x) dx^jdx^k $, $\lambda=1,2$ be two Riemannian metrics in $\R^3$, both admissible $($Definition \ref{admi}$)$ with respect to an open $U \subset W$. 
 Let $\Upsilon \subset \partial W$ be a stable part of the boudary $($Definition \ref{stabbel}$)$ with respect to $U$ for both metrics $g_\lambda$, $\lambda=1,2$
 and 
assume that 
 \[
g_1|_x = g_2|_x  \quad \forall x \in \Upsilon.
\]
Consider two differential operators $P_\lambda \in \mpsc{2}$, $\lambda=1,2$ of the form
\[
P_{\lambda}  (x,t,D)   = \left(  \partial_t^2 -  g_\lambda^{jk} (x)   \frac{\partial}{\partial x^j} \frac{\partial}{\partial x^k} \right) I   + F_\lambda (x,D_x) , 
\]
where the entries $(F_\lambda)^k_j $, $j,k=1,2,3$ of $F_\lambda$ are first order operators on $\R^3$, and sources $S_\lambda(x,t)  := S_{\lambda,j}( x,t ;z,\theta,\beta) dx^j$ that satisfy (\ref{meq1}-\ref{meq3}). 
%
Let $u_\lambda(x,t) := u_{\lambda,j}( x,t ;z,\theta,\beta) dx^j $  solve 
\begin{align}
& P_\lambda u_\lambda  - S_\lambda   \in C^\infty( \R^4; \R^3)   ,   \\
& u_{\lambda}   |_{\R^3 \times (-\infty,-T]} \in C^\infty( \R^3 \times (-\infty,-T] ; \R^3). 
\end{align}
for some $T=T_{z,\beta,\theta}>0$. 
Assume that for every $z \in U$ there is an open nonempty subinterval 
\[
\mathcal{I}_z \subset   \Big(  \beta_{\partial W} (z)   , 1 \Big), \quad \beta_{\partial W}(z) := \max\limits_{\lambda =1,2} \beta_{\partial W} (z;g_\lambda), 
\]
such that the data on $\mathcal{I}_z$  coincide for both metrics, i.e.
\[
 \bigcup_{j=1,2,3} WF (   u_{\lambda=1,j}  )|_{T(\Upsilon \times \R) }  =  \bigcup_{j=1,2,3} WF (   u_{\lambda=2,j} )|_{ T(\Upsilon \times \R)}, \quad \forall \beta \in \mathcal{I}_z, 
\]
for every $(z,\theta) \in U \times \S^2$. 
%
Then, $g_1|_U = g_2|_U$. 

\end{theorem}

\begin{remark}\label{labne}
Provided that  \eqref{memma2} in Definition \ref{admi} is substituted by the stronger condition  
\[
|\theta|_g(z) >1, \quad \forall (z,\theta) \in \overline{U} \times \S^2,
\] 
 the compactness of $\overline{U}$ yields $\sup_{z\in U} \beta_{\partial W}(z) < 1$ and we can choose $\mathcal{I}_z = \mathcal{I}$ independently on $z\in U$ for some $\mathcal{I} \subset (0,1)$ with $\sup \mathcal{I} < 1$. 
This means that it suffices to make measurements with moderate speeds $\beta$ to obtain the result above. 
\end{remark}

\subsection{Background and Previous Work}

%

Cherenkov radiation is named after Pavel Cherenkov who shared the 1958 Nobel Price in Physics for its discovery. The historical background is briefly discussed in \cite{Bolotovskii:2009}. 
The phenomenon is applied in particle detectors (see the review \cite{YPSILANTIS199430} and references therein) 
and in detection of biomolecules. Applications of Cherenkov luminescence in medical imaging and radiotherapy are studied extensively. \cite{RuggieroAlessandro2010Clio}, \cite{invivo}, \cite{lumi}, \cite{PratxGuillem2018IClb}, \cite{27ad72c5d1da42deb078fcc0ccdb8caa}
For Cherenkov radiation as a quantum effect in vacuum, see \cite{osti_6542034}. 
Cherenkov radiation is analogous to so-called Askaryan radiation, first observed experimentally in 2000. \cite{Saltzberg:2000bk}
See \cite{rocksalt}, \cite{PhysRevLett.96.171101}, \cite{Lehtinen:2003xv} for subsequent research. 

Inverse problems related to Cherenkov or Askaryan radiation appears to be explored very little, if at all. 
There are, however, several studies on inverse problems for particle flows and analysis on wave fronts of propagating waves that are related to this paper or use analogous microlocal techniques. These include 
inverse problems related to interaction of waves \cite{MR3800791}, \cite{MR3800791}, \cite{feizmohammadi2020inverse}, \cite{MR3995093}, \cite{chen2019detection}, \cite{Chen:2020zdb}, \cite{MR3812861}, \cite{hintz2021inverse}, 
single scattering inverse problems  \cite{MR2654781}, \cite{dehoop2020foliated}, \cite{MR3310276}, \cite{doi:10.1137/130931291},  and 
inverse problems for particle models \cite{MR2495267}, \cite{MR1391523}, \cite{MR1413417}, \cite{MR4108206}, \cite{MR2094548}, \cite{MR2180309}.
For studies with relativistic particles, see \cite{MR3810149}, \cite{MR3411742}, \cite{MR3272336}, \cite{MR3139296}.
%
%
%
%
%
%
Some of the methods applied in this work have points of resemblance in recovery of singularities in scattering theory which is a well-studied topic, especially in the Euclidean context. \cite{Greenleaf-Uhlmann}, \cite{MR2578561}, \cite{MR1260297}, \cite{MR1084969}, \cite{MR1842046}, \cite{MR1662267}, \cite{MR1617700}, \cite{MR2942743}

We refer to \cite{FIO1},\cite{FIO2},  \cite{MR610185}, \cite{10.1007/BFb0074191} \cite{Melrose-Uhlmann}, \cite{MR1040963} for theory on FIOs. 
See also the preceding works \cite{MR97628}, \cite{MR115010},  \cite{MR0265748}, and \cite{Maslovi1965, MR1109496}.  

\subsection{Acknowledgements}
I would like to thank Academy Professor Matti Lassas for his help and guidance during the project.
This work was financially supported by the ATMATH Collaboration and Academy of Finland (grants 336786 and 320113).

\subsection{Illustrative Example}\label{brieffi}
We give a short introduction to Cherenkov radiation in flat geometry which is the space-time that corresponds to homogeneous isotropic medium. 
The model in this example is trivial from a geometric point of view. 

Let us consider the following forward propagating scalar wave:
\begin{align}
&( \partial_t^2 - k^2 \Delta ) u (x,t)  = \chi(x) S_0(x,t)   , \quad  u \in \mathcal{D}' ( \R^4) , \quad 0<k<1, \\ 
& u|_{\R^3 \times (-\infty ,T] } \in C^\infty( \R^3 \times (-\infty ,T] ) ,  \\ 
&S_0  :=  \delta(x^1- t \beta) \delta(x^2) \delta(x^3). 
\end{align}
($z= 0$, $\theta= (1,0,0)$) Here $\Delta = \delta^{jk} \frac{\partial}{\partial x^j} \frac{\partial}{\partial x^k}$ and $\chi \in C^\infty_c(\R^3)$, $\chi|_{B(0,1)} = 1$. 
The characteristic manifold $L\R^4$ for the wave operator is 
\[
\text{Char} ( \partial_t^2 - k^2 \Delta) = \{ (x,t,\xi,\omega) \in \R^4 \times ( \R^4 \setminus \{0\} ) : \omega^2 = k^2 |\xi|^2  \}. 
\]
The wave front set of the wave splits (see e.g. \cite{Duistermaat}) into the static and propagating parts (cf. the discussion in Section \ref{dfsoerew}). That is;
\[
\begin{split}
WF(u) \subset WF ( \chi  S_0 ) \  \cup  \  \Lambda_f \circ  WF( \chi S_0 )  
\end{split}
\]
where 
 $\Lambda_f  \subset (T^* \R^4 \times T^* \R^4 )  \setminus \{0,0\} $ stands for the forward propagating flow-out canonical relation generated by $( \partial_t^2 - k^2 \Delta)$, i.e. 
\[
\Lambda_f = \bigg\{  \Big(   x \mp   r k v(\xi)  , t  + r  ,  \xi , \pm k |\xi|  \  ; \ x,t,\xi,  \pm k |\xi|    \Big) :   r \geq  0, \ (x,t,\xi) \in \R^4 \times \R^3 \bigg\} , \quad v^j(\xi) : =   \frac{\xi_k \delta^{jk} }{ |\xi| } .
\] 
Here the composition $\Lambda_f \circ  WF( \chi S_0 ) $ stands for
\[
\Lambda_f \circ  WF( \chi S_0 )  := \{ (x,t,\xi,\omega) : \big((x,t,\xi,\omega), (y,s,\eta,\rho) \big) \in \Lambda_f, \  (y,s,\eta,\rho) \in  WF( \chi S_0 )  \} . 
\]
One checks that 
\[
\begin{split}
WF ( \chi  S_0 )  &= N^* K(0,(1,0,0) ,\beta) \  \cap  \ (  \text{supp} ( \chi ) \times \R  \times \R^4 ) \\ 
&= \big\{   (t\beta,0,0, t   ,   \xi,-\beta \xi_1)  : t\in \R , \  \xi \in \R^3 \setminus \{0\}    \big\}  \  \cap  \ (  \text{supp} ( \chi ) \times \R  \times \R^4 ). 
\end{split}
\]
For $\beta < k$ we deduce 
\[
 | \beta \xi_1 | \leq \beta |\xi| < k |\xi |  \quad  \Rightarrow \quad   - \beta \xi_1  \neq \pm k |\xi| \quad \Rightarrow \quad     WF ( \chi  S_0 ) \cap  \text{Char} ( \partial_t^2 - k^2 \Delta)   = \emptyset. 
\]
Thus, $\Lambda_f \circ WF ( \chi  S_0 ) = \emptyset$ and we obtain 
\[
WF(u) \subset  WF ( \chi  S_0 ) , \quad \text{for} \quad  \beta < k. 
\]
That is; in the regime $\beta< k$ no singularities radiate from the source. 
 If $\beta > k$, then there exist bicharacteristics through $WF ( \chi  S_0 )$ and the composition  $ \Lambda_f \circ WF ( \chi  S_0  ) $ is non-trivial. In fact, the projection of it to $\R^4$ is a union of the sets 
\[
\begin{split}
 Q_{x,t} =  \bigg\{ (x,t) + r  (  k  \tilde\theta  , 1 )  : \tilde\theta \in \S^2, \   \tilde\theta^1 \left( \frac{\beta^2}{k^2} - 1 \right)^{\frac{1}{2}} = |( \tilde\theta^2, \tilde\theta^3)| , \ r >0 \bigg\} .
 \end{split}
\]
over points $(x,t)$ that satisfy $x \in \text{supp} ( \chi) , \ x= (t \beta,0,0)$.  
Each of these manifolds corresponds to propagating singularities generated at the associated point $(x,t)$ by the source.   
The equation $\tilde\theta^1 \left( \frac{\beta^2}{k^2} - 1 \right)^{\frac{1}{2}} = |( \tilde\theta^2, \tilde\theta^3)|$ corresponds to the Frank-Tamm formula in physics. 
For $\beta = k$ the cone collapses into a line. This particular case is problematic in the microlocal framework as it fails to satisfy conditions required in the standard FIO calculuses. 

\subsection{Geometric Example: Electromagnetic Waves in Anisotropic Material with Scalar Wave Impedance}\label{eleex}
The objective of this example is to explicate how the equation
\[
\left(\partial_t^2 - g^{jk} \frac{\partial}{\partial x^j} \frac{\partial}{\partial x^k} \right) I v    + Fv  \equiv  S  \mod C^\infty(\R^4;\R^3) , 
\]
is linked to electromagnetism. 
Consider a system of moving charges in dielectric anisotropic medium with smoothly varying inhomogeneities. The model is described by the Maxwell's equations
\begin{align}
 \curl E(x,t) + \partial_t B(x,t) &= 0, \label{11aaekama} \\
 \dive  B(x,t)  &= 0, \\
 \curl H (x,t)  - \partial_t  D(x,t) &= J(x,t),\label{11aakolmama} \\
 \dive D(x,t)   &= \rho(x,t) , \label{11aavikama}  
 \end{align}
 where $E(x,t)= E^j(x,t) \frac{\partial}{\partial x^j}$, $B(x,t) =B^j(x,t)  \frac{\partial}{\partial x^j}$, $D(x,t) = D^j(x,t) \frac{\partial}{\partial x^j}$, $H(x,t) = H^j(x,t)  \frac{\partial}{\partial x^j}$, $\rho(x,t)$, and $J(x,t) = J^j(x,t)  \frac{\partial}{\partial x^j}$ are the associated electric field, magnetic field, electric displacement field, auxiliary magnetic field, charge density, and current density, respectively. 
Properties of the medium are encoded in the smooth $(1,1)$-tensors 
\begin{align}
\varepsilon(x) &= \varepsilon_j^k(x) dx^j \otimes  \frac{\partial}{\partial x^k}, \\
\mu(x) &= \mu_j^k(x) dx^j \otimes \frac{\partial}{\partial x^k} 
\end{align}
which give connections between the fields via 
 \begin{align}
 D(x,t) &= \varepsilon (x) E(x,t),  \\
 B(x,t)& =\mu(x) H(x,t).
\end{align}
For lossless, optically inactive materials it is reasonable to require the tensors $\varepsilon(x)$ and $\mu(x)$ to be real valued, symmetric and non-degenerate. 
In addition, we assume that the material has scalar wave impedance. By definition, this means that there is a smooth scalar function $\alpha(x)$ such that $\mu(x) = \alpha^2(x) \varepsilon(x)$. 
 Consider a charge carrier (e.g. an electron) or some other non-smooth charge perturbation moving along a straight line in $\R^3$ through anisotropic dielectric medium with smoothly varying inhomogeneities. 
Let $\theta \in \S^2$, $\beta \in (0,c)=(0,1)$, and $z\in \R^3$ be the direction of motion, speed, and location at $t=0$ for the moving charge perturbation, respectively. Again, we work in a natural unit system that has $c=1$ as the speed of light in vacuum. 
The perturbation signal is encoded in the charge density $\rho(x,t)$ as an oscillatory superposition
\begin{align}
\rho(x,t) =  \int_{\R^3} e^{i (x-z-t\beta \theta)\cdot \xi} a(x,\xi) d\xi,
 \label{fdso02} 
\end{align}
where the amplitude has an asymptotic development 
\begin{equation}\label{asy6}
a(x,\xi) \sim  a_m(x,\xi)+ a_{m-1}(x,\xi) + \cdots 
\end{equation}
 into functions $a_j \in S^{j} ( \R^3 \times (\R^3 \setminus \{ 0 \}) )$. 
For example, a single electron in vacuum corresponds to the density 
\[
\rho_{e^-}(x,t) = -e \gamma \delta(x-z-\beta t \theta)=  -e \gamma \int_{\R^3} e^{i ( x-z-t\beta \theta ) \cdot \xi  } d\xi,
\]
where $e$ stands for the elementary charge and $\gamma := \frac{1}{\sqrt{1-\beta^2}}$ is the Lorentz factor. 
 It follows (see e.g. \cite{Duistermaat}) that the distribution \eqref{fdso02} is smooth outside the trajectory $K(z,\theta,\beta):=\{ (x,t) : x= z + t\beta \theta \}$ and the order of singularity at each point corresponds to decay of $a$ with respect to $\xi$. 
 The signal is a moving point-like singularity even if all the charge is not necessarily concentrated at a single point. 
 The current density obeys the continuity equation (conservation of charge)
 \[
 \dive J(x,t) + \partial_t \rho (x,t) = 0
 \]
 which follows from \eqref{11aakolmama}, \eqref{11aavikama} and the fact that $\dive \ \curl = 0$. 
 By substitution one checks that the equation is solved by 
 \[
 J(x,t) = J_I  + J_S(x,t), \quad J_I =  \left(  \int_{\R^3} e^{i (x-z-t\beta \theta)\cdot \xi} b(x,\xi) d\xi  \right) \theta
 \]
 where $b(x,\xi) \sim \sum_{j=-m}^\infty  b_{-j}(x,\xi)$ is defined recursively by 
 \begin{align}
 &b_m (x,\xi) = \beta a_m(x,\xi), \\
 & i \xi \cdot \theta b_{j-1}(x,\xi)   + \theta \cdot \nabla b_{j}(x,\xi) = i \beta \xi \cdot \theta a_{j-1} (x,\xi), \quad j=1,2,3,\dots
 \end{align}
 and $J_S(x,t)$ is any solenoidal field. For simplicity, let us assume that $J_S(x,t)$ is smooth. 
 Define a Riemannian metric $g$ by 
 \[
g^{jk}(x) :=    \frac{\delta^{jl}  \varepsilon_l^k(x)}{ \alpha^2(x) \det(\varepsilon(x))}.
\]
and set $\delta_\alpha := (-1)^k * \alpha d \alpha^{-1} *$ where $*$ is the Hodge star on $k$-forms with respect to $g$. 
Under the requirements above the electric field $E_j = \delta_{jk} E^k$ satisfies (See \cite{Kurylev-Lassas-Somersalo2006}) 
\[
( \partial_t^2 - \Delta_\alpha)^k_j E_k = S_j, \quad j=1,2,3,
\]
where $\Delta_\alpha := - d \delta_\alpha - \delta_\alpha d$ and 
\begin{align}
S_k =&  \frac{1 }{ \alpha^2 \det \varepsilon(x) } \int_{\R^3} e^{i(x-z-t\beta\theta) \cdot \xi} c_k (x,\xi) d\xi , 
\\
 c_k(x,\xi)  \equiv& i( \xi_k-  \beta^2 (\theta \cdot \xi)  g_{jk} \theta^j  ) a_m(x,\xi) \mod  S^{m-1}. 
\end{align}
By applying Leibniz' rule, we deduce that the standard codifferential $\delta = -*d*$ differs from $\delta_\alpha$ by only terms  of order $0$. For instance, for a 1-form $\omega$ we have that 
\[
\delta_\alpha \omega = - * \alpha d \alpha^{-1} *  \omega = - * d *  \omega - * \frac{d\alpha}{\alpha} \wedge *\omega =  \delta \omega -  g^{jk} \frac{\partial_k  \alpha}{\alpha}   \omega_j . 
\]
Thus, 
 the leading terms in $\Delta_\alpha$ coincide with the ones in the Hodge Laplacian $-d\delta-\delta d$. As shown in \cite[Lemma 2.8]{MR3586566}, the leading part in the Hodge Laplacian for 1-forms is $(\partial_t^2 - g^{jk} \partial_j \partial_k) I$. In conclusion,  we have that the electric field satisfies the preferred equation:
\[
\begin{pmatrix}
\partial_t^2 - g^{jk} \partial_j \partial_k  & 0 & 0\\
0 & \partial_t^2 - g^{jk} \partial_j \partial_k   & 0\\
0 & 0& \partial_t^2 - g^{jk} \partial_j \partial_k  
\end{pmatrix}
 \begin{pmatrix} 
 E_1 \\
 E_2\\ 
 E_3\\ 
 \end{pmatrix}
 + \begin{pmatrix}
 F^1_1   & F^1_2 & F_3^1\\
F^2_1 & F^2_2  & F^2_3\\
F^3_1 & F^3_2& F^3_3
\end{pmatrix}
 \begin{pmatrix} 
 E_1 \\
 E_2\\ 
 E_3\\ 
 \end{pmatrix}
  =  \begin{pmatrix} 
 S_1 \\
 S_2\\ 
 S_3\\ 
 \end{pmatrix},  
\]
where $F_j^k \in \Psi^1_{cl}(\R^3)$, $j,k=1,2,3$. 


\section{Preliminaries}\label{premmi}


\fbox{
Henceforth the symbol $\partial_j$ with integer $j$ stands for the spatial partial derivative $ \frac{\partial}{\partial x^j}$.
}

\subsection{Symbols }

Let $X$ be a smooth manifold of dimension $n$. 
The class $S^m_{\rho,\delta}(X \times \R^k) $ of symbols of order $m \in \R$ and type $(\rho,\delta) \in [0,1] \times [0,1]$  is defined as the space of functions $a \in C^\infty(X \times \R^k)$ that satisfy the following: 
For each compact $K \subset X$, $\alpha \in \N^n$ and  $\beta \in \N^k$ there exists $C = C_{K,\alpha,\beta,a}\geq 0$ such that
\[
| \partial_x^\alpha \partial_\xi^\beta a(x,\xi) | \leq C (1+ |\xi|)^{m-\rho |\beta|+\delta |\alpha|}, \quad \forall (x,\xi) \in K \times \R^k.
\]
Only symbols of type $(1,0)$ are considered and hence the notational simplification $S^m (X \times \R^k) := S_{1,0}^m (X \times \R^k)  $. It follows from the definition that $S^{m}_{\rho,\delta}(X \times \R^k)  \subset S^{m'}_{\rho,\delta}(X \times \R^k) $ for $m\leq m'$ and we define $S_{\rho,\delta}^{-\infty }(X \times \R^k)$ as the limit set $S^{-\infty }_{\rho,\delta} (X \times \R^k) := \bigcap_{m \in \R} S^m_{\rho,\delta}  (X \times \R^k)$. 
It is common to treat symbols modulo a residual term in $S^{-\infty}_{\rho,\delta}$. 
For a sequence of symbols $a_j \in S^{m_j}_{\rho,\delta}(X \times \R^k) $, $j=1,2,3,\dots$, with $m_j\searrow -\infty$ there exists $a \in S^{m_1}_{\rho,\delta}(X \times \R^k) $, unique up to a term in $S^{-\infty} (X \times \R^k)$, such that
\[
a- \sum_{j=1}^{l} a_j \in S^{m_{l+1}}_{\rho,\delta}(X \times \R^k) , 
\]
for every $l=1,2,3,4,\dots$  This is denoted by $a(x,\xi) \sim  \sum_{j=1}^{\infty} a_j(x,\xi)$ and referred to as an asymptotic expansion or asymptotic development of $a$. 

The set of classical symbols  $S_{cl}^m(X \times \R^k)\subset S^m(X \times \R^k)$ is defined as the elements $a \in S^m(X \times \R^k)$ of the form
\begin{equation}\label{asdiieierw234234}
a(x,\xi)  \sim   \sum_{j=1}^{\infty} (1- \chi(\xi) ) a_{m-j}(x,\xi) , 
\end{equation}
where $\chi \in C^\infty_c ( \R^k)$ takes value $ 1$ in a neighbourhood of $0$ and $a_{m-j} \in C^\infty(X \times (\R^k\setminus \{0\}) )$ is positively homogeneous of degree $m-j$ with respect to the variable $\xi \in \R^k \setminus \{0\}$. 
Even though positively homogeneous functions are usually not smooth, they are treated as symbols by omitting a smoothing near the origin. The notation $a(x,\xi)  \sim   \sum_{j=1}^{\infty}  a_{m-j}(x,\xi)$ for positively homogeneous $a_{m-j}(x,\xi)$ is used also for referring to \eqref{asdiieierw234234}. 
The choice of the compactly supported cut-off function $\chi$ around $0$ is not significant as changing the function within contributes only by a term in $S^{-\infty} (X \times \R^k)$. 
See e.g. \cite{Grigis-Sjostrand} for a nice introduction to the topic. 

\subsection{Lagrangian Distributions and Fourier Integral Operators} \label{lagger}

A submanifold $\Lambda$ of a symplectic manifold $(M,\sigma)$ of dimension $2n$ is called Lagrangian if $\dim \Lambda = n$ and $\sigma_q (v,w) = 0$, for every $v,w\in T_q \Lambda$ and $q\in \Lambda$. 
The cotangent bundle $T^*X$ of a smooth manifold $X$ is a symplectic manifold with the canonical symplectic structure (see \cite[\textsection5]{Grigis-Sjostrand}).   
The space $I^r(\Lambda) = I^r( X; \Lambda)$ of Lagrangian distributions of order $r$ on $X$ associated with a conic Lagrangian manifold $\Lambda \subset T^*X \setminus \{0\}$ consists of distributions $u \in \mathcal{D}'(X)$ that can be expressed as a locally finite sum of oscillatory integrals of the form
\begin{equation}\label{89834234}
I(\varphi,a) = \int_{\R^k } e^{i\varphi(x,\xi)} a(x,\xi) d\xi  , \quad a\in S^{r-\frac{k}{2} + \frac{n}{4} }(X \times \R^k), 
\end{equation}
where $\varphi \in C^\infty( X \times \R^k) $ is a non-degenerate phase function, homogeneous of degree 1 with respect to $\xi$, such that that the manifold $\Lambda$ coincides locally with the set $\Lambda_\varphi = \{ (x,d_x\varphi(x,\xi) ) :   d_\xi \varphi(x,\xi) = 0 , \ \xi \neq 0, \ x \in X   \}$ (see e.g. \cite{Duistermaat} for details). 
For $u \in I^r( X; \Lambda)$ the wave front set satisfies $WF(u) \subset  \Lambda$ and the asymptotic behaviour of the symbol defines the degree of regularity at given conic neighbourhood. 
For instance, if $\chi a \in S^{-\infty}(X \times \R^k)$ where $\chi \in C^\infty (X \times \R^k)$ is positively homogeneous of degree $0$ with respect to $\xi$, then the wave front set of the oscillatory integral \eqref{89834234} does not meet a conic neighbourhood inside the support of $\chi$. 
As a special case, we have $I^{-\infty} ( X; \Lambda) \subset C^\infty( X)$. If two symbols $a$ and $b$ satisfy $a-b \in S^{-\infty} ( X \times \R^k)$, i.e. $a\equiv b \mod S^{-\infty} ( X \times \R^k)$, then  $I(\varphi,a) - I(\varphi,b) = I(\varphi,a-b) \in C^\infty(X)$, that is, $I(\varphi,a) \equiv I(\varphi,b) \mod C^\infty(X)$ for the associated oscillatory integrals.
We do not distinguish between distributions that differ by a smooth term.

A special class of Lagrangian distributions is obtained by assuming that $\Lambda$ equals $N^*M$,  the conormal bundle of a submanifold $M \subset X$. Such elements are referred to as distributions conormal to $M$.
For example, the delta distribution $u(x,t) = \delta(t-x)$ on $ \R^2$ is  conormal to the diagonal $\{(x,t)  \in \R^2 : t = x\}$. 
 The class of distributions conormal to $M$ is often denoted by $I^m ( M)$, where $m$ refers to the order of the symbol. Thus, $I^m ( M) = I^{m+\frac{k}{2} - \frac{n}{4}}(X; N^*M)$, where $k$ is the codimension of $M$ in $X$.

Considering a Lagrangian distribution $f$ on $X\times Y$ as a kernel gives an operator
\[
F=F_f : C^\infty(Y) \rightarrow \mathcal{D}'(X), \quad  \langle F \phi , \psi \rangle := \langle f,   \psi \otimes \phi \rangle , \quad \phi\in C^\infty(Y), \quad  \psi \in C^\infty (X).
\] 
Functions of this form are called (standard) Fourier integral operators.\footnote{An operator and the kernel of it are often treated as the same object. }
For a closed cone $\Gamma \subset T^*Y$ that does not meet $WF_Y'(F) := \{ (y,\eta) : \exists x\in X \ \text{such that} \ (x,0;y,\eta) \in WF(f) \} $ the definition can be extended to  $\mathcal{E}_\Gamma'(X) = \{ v \in \mathcal{E}'(X) : WF(v) \subset \Gamma \}$ and further to $\mathcal{D}_\Gamma'(X) = \{ v \in \mathcal{D}'(X) : WF(v) \subset \Gamma \}$ if the operator is properly supported. See e.g. \cite[Corollary 1.3.8]{Duistermaat} for details. 
If $f\in I^m ( X \times Y;\Lambda)$ we say that the associated operator $F_f$ lies in $I^m ( X,Y;\Lambda')$. The manifold $\Lambda'$ refers to $\Lambda' :=\{ (x,\xi;y,\eta) : (x,\xi;y,-\eta) \in \Lambda\}$ which is Lagrangian with respect to the symplectic form $\sigma_X- \sigma_Y$.  
Such a manifold is referred to as the canonical relation. Conversely, if a manifold $\Lambda \subset T^*X \times T^*Y$ is a canonical relation, then $\Lambda'$ is Lagrangian manifold for the standard symplectic form $\sigma_X+ \sigma_Y$ and the kernel is a Lagrangian distribution on $X \times Y$. 
 A Fourier integral operator with a smooth kernel defines a regularizing map $\mathcal{E}'(X) \rightarrow C^\infty(X)$ which is usually identified with the zero element.  

An operator that admits the diagonal $ \text{diag}(T^*X\setminus \{0\})$ as a canonical relation is called pseudo-differential operator. We denote by $\Psi^m(X)$ the class of pseudo-differential operators of order $m$ on $X$. Pseudo-differential operators with classical symbols is denoted by $\Psi^m_{cl}(X)$.
All the pseudo-differential operators considered in this paper have $X=\R^n$. 
By placing a smooth cut-off $\chi(x-y)$ to the kernel (see \cite[Remark 3.3]{Grigis-Sjostrand}) a pseudo-differential operator can be identified with a properly supported one.

\subsection{Transversal Intersection Calculus}
A composition of two Fourier integral operators is not necessarily well defined as a Fourier integral operator. Sufficient conditions together with associated composition calculus, often referred to as transversal intersection calculus, was developed in \cite{FIO1}, \cite{FIO2} by H\"ormander and Duistermaat.  
Perhaps the most demanding one of the conditions for a composition $F \circ G$ of two operators $F\in I^m(X,Y;\Lambda_1)$ and  $G\in I^{m'}(Y,Z;\Lambda_2)$ with canonical relations $\Lambda_1$ and $\Lambda_2$ to be admissible in the framework is that the product $\Lambda_1 \times \Lambda_2$ intersects the manifold $T^*X \times \text{diag}T^*Y \times T^*Z$ transversally. 
Provided that the required conditions are satisfied, we have
\begin{align}
&F\circ G \in I^{m+m'} (X,Z ; \Lambda_1 \circ \Lambda_2),  \\
& \Lambda_1 \circ \Lambda_2  := \big\{ (x,\xi; z, \sigma) \in T^* X \setminus \{0\} \  \big| \  \exists (y,\eta)  :  (x,\xi; y, \eta) \in \Lambda_1 , \  (y,\eta;z, \sigma) \in \Lambda_2 \big\}.
\end{align} 
Moreover, 
\begin{equation}\label{symboli123}
\sigma_{m+m'}(F \circ G) ( x,\xi; z,\sigma)  = \sum_{(y,\eta) \in \Gamma} \sigma_m(F) (x,\xi ; y,\eta) \sigma_{m'}(G)(y,\eta;z,\sigma), 
\end{equation}
where $\sigma_m \in S^m/S^{m-1}$ refers to the principal symbol (i.e. the leading term in an asymptotic development) on the canonical relation and 
 $\Gamma= \Gamma_{x,\xi;y,\eta} :=  \{ (y,\eta) :  (x,\xi ; y,\eta) \in \Lambda_1, \ (y,\eta;z,\sigma) \in \Lambda_2 \}$. 
 The diagonal acts as an identity element on canonical relations:
\[
\text{diag} (T^*X\setminus \{0\}) \circ \Lambda_2 = \Lambda_2, \quad   \Lambda_1 \circ \text{diag} (T^*Y\setminus \{0\}) = \Lambda_1,
\]
Considering Lagrangian distributions on $Y$ as Fourier integral operators on $Y \times Z$ with trivial $Z$ one deduces 
\[
F  u \in I^{m+r} ( X ; \Lambda \circ \tilde\Lambda), \quad \text{for} \quad F \in I^m (X,Y;\Lambda), \quad u \in I^r(Y;\tilde\Lambda),
\]
provided that the required conditions hold. The element $\Lambda \circ \tilde\Lambda$ above refers to 
\[
\Lambda \circ \tilde\Lambda := \{ (x,\xi) : \exists(y,\eta) \in \tilde\Lambda; \  (x,\xi;y,\eta) \in \Lambda\}. 
\]
For a pseudo-differential operator $A$ we obtain  
\[
A  u \in I^{m+r} ( X ; \text{diag} (T^*\R^n \setminus \{0\}) \circ \tilde\Lambda)  ) = I^{m+r} ( X ;  \tilde\Lambda  )
\]
(cf. microlocality $WF( A  u) \subset WF(u)$).

\subsection{Distributions associated with a pair of Lagrangian manifolds.}\label{dfsoerew}
 We also consider the class $I^r( X; \Lambda_0,\Lambda_1)$ of Lagrangian distributions associated with a pair of cleanly intersecting Lagrangian manifolds $\Lambda_0,\Lambda_1 \subset T^*X$ instead of a single manifold. 
 Calculus for operators with kernels of this form was developed by Melrose and Uhlmann in \cite{Melrose-Uhlmann}.  The theory is needed for describing parametrices  of pseudo-differential operators.  
 More precisely, a parametrix $Q$ for a pseudo-differential operator $P$ of real principal type on $X$ is associated with the pair 
 \[
 (\text{diag}(T^*X \setminus \{0\}) , \Lambda_P) \subset (T^*(X) \times T^*(X)) \times (T^*(X) \times T^*(X) ), 
 \]
 of canonical relations  
 where $\Lambda_P$ consists of pairs $((x,\xi),(y,\eta)) \in (T^*X \times T^*X) \setminus \{0,0\}$ that lie in a same bicharacteristic. 
It is shown in \cite[Proposition 2.1]{Greenleaf-Uhlmann} that  for a distribution $u$ in the class $  I^{m} ( X; \Lambda)$, where the Lagrangian manifold $\Lambda \subset T^*X$ intersects the characteristic manifold of $P$ transversally and each bicharacteristic of $P$ intersects the manifold a finite number of times, the composition 
 $Qu$ is a Lagrangian distribution that corresponds to the pair 
 \[
 (\text{diag}(T^*X \setminus \{0\}) \circ \Lambda \  , \  \Lambda_P \circ \Lambda ) = ( \Lambda, \Lambda_P \circ \Lambda ) .
 \] 
Moreover, the identity \eqref{symboli123} extends to $ (\Lambda_P \circ \Lambda) \setminus \Lambda$.

\subsection{Matrix Notation}

We define the class $  \mps{m} $ as $3\times 3$ matrices $A = [A_j^k]_{j,k=1,2,3}$ of pseudo-differential operators $ A_j^k \in \Psi^m ( \R^4)$, $j,k=1,2,3$, of degree $m$ on $\R^4$ as entries. The symbols of $A^k_j$, $j,k=1,2,3$ form an element in 
$M_3 ( S^m( \R^4 \times \R^4 ) ) := \{ [a^k_j]_{j,k=1,2,3} : a_j^k \in S^m( \R^4 \times \R^4 ) \}$ of matrices with entries in the symbol class $S^m( \R^4 \times \R^4)$. 
%
Each $ A \in  \mps{m}$ operates on a vector-valued distribution $ S= (S_1,S_2,S_3)$, $S_j \in \mathcal{E}'(\R^4) $, 
$j=1,2,3$, in the obvious way:
\[
(A S)_j  = A^k_j S_k \in  \mathcal{D}'(\R^4), \quad j=1,2,3. 
\]
This is well defined also for $S_k \in \mathcal{D}'(\R^4) $ if the entries $A^k_j$ are properly supported.     
In addition, we define the composition $AB \in  \mps{m_1+m_2}$ of  $A \in \mps{m_1}$ and  $ B \in  \mps{m_2}$ by
\[
(AB)^k_j := A^l_j B_l^k , \quad j,k=1,2,3. 
\]
Any scalar operator $ L  \in   \Psi^l ( \R^4)$ is identified with the diagonal matrix $[L\delta^k_j]_{j,k=1,2,3}$, where $\delta^k_j$ stands for the Kronecker delta. That is; 
\[
(L S)_j  =  LS_j \in  \mathcal{D}'(\R^4), \quad j=1,2,3. 
\]
for $ S= (S_1,S_2,S_3)$, $S_j \in \mathcal{D}'(\R^4) $,  
and 
\[
(L A)^k_j = LA^k_j \in  \mps{m+l} ,
\]
for $A \in \mps{m}$.

\section{Microlocal Methods}\label{micromikko}
This section is based on the techniques developed in \cite{FIO1},\cite{FIO2},  \cite{Melrose-Uhlmann} and \cite{Greenleaf-Uhlmann} for scalar operators of real principal type. The results here are modifications of the original ones. Some of the proofs are moved to Appendix \ref{appe1} as they do not differ significantly from the ones for scalar-valued operators.
The objective is to prove Proposition \ref{pinheiro} which is the main tool in this paper.

We use the notation $\pm$ to indicate a sign $+/-$ that can be chosen to be either $+$ or $-$. We also define $\mp := - \pm$. 
For a conic neighbourhood $\mathcal{V} \subset T^*\R^4 $  and  $X,Y \in  M_3 ( \Psi^m (\R^4) )$ we say that 
\[
X \equiv Y \quad  \text{microlocally in} \quad  \mathcal{V}   
\]
if 
\[
\mathcal{V} \cap WF( X^k_j - Y^k_j)  = \emptyset, \quad  \forall j,k=1,2,3,
\]
where $WF(A)$ refers to the wave front set of a pseudo-differential operator, that is, $WF'(A) = \text{diag} ( WF(A) )$ as a Fourier integral operator. 
Let us begin by factoring a vector-valued operator $P \equiv ( \partial_t^2 - g^{jk} \partial_k \partial_j ) I + F$ into first order operators. For a proof, see the appendix. 
\begin{lemma} \label{equdfkss}

Consider a differential operator 
$P \in M_3( \Psi^2_{cl}(\R^4) )$ of the form
\[
P =   (\partial_t^2 - g^{jk}\partial_j \partial_k ) +  F ,
\]
where $g$ is a Riemannian metric and $F  =    F^k_{j}(x,D_x )  \  \partial_k \otimes   dx^j$ is a stationary first order operator. 
%
Then, for every $\epsilon>0$ there are operators 
$W=W_\pm  \in M_3 ( \Psi^1_{cl}(\R^4) ) $, $Z=Z_\pm \in \mpsc{1}$,  
and $h=h_\pm \in M_3 ( \Psi^0_{cl}(\R^4) )$ 
such that 
\[
 P  \equiv  (  D_t  - W ) Z , 
 \]
 and 
 \begin{align}
&W  \equiv  \pm   | D_x |_g  + h  ,  \\
&Z   \equiv - (  D_t  \pm   |D_x|_g ) - h   ,
\end{align}
microlocally in $ \mathcal{V} = \{ (x,t ; \xi,\omega) \in T^*\R^4 \setminus \{0\} : |\omega| < \epsilon |\xi| \}$. 
\end{lemma} 
  Define $L^\pm \R^4 = \{ (x,t,\xi,\omega) \in T^*\R^4 : \omega= \pm |\xi|_g \}$, where $g$ is a Riemannian metric. 
Throughout this section we consider a non-zero covector $(x_0,t_0,\xi_0, \omega_0) $  in  $T^*\R^4$ with  $\xi_0 \neq 0$ and $\omega_0 \neq \mp |\xi_0|_g$ for a given sign $\pm \in \{+,-\}$  and  a fixed homogeneous canonical transformation 
\begin{equation}\label{fixh}
(y,s,\eta,\rho)  \mapsto \mathcal{H}(y,s,\eta,\rho) \in T^*\R^4 , 
\end{equation}
defined on a conic neighbourhood of the covector $(y_0,s_0,\eta_0,\rho_0)$ with 
\[
\mathcal{H}^{-1}(x_0,t_0,\xi_0,\omega_0) =  (y_0,s_0,\eta_0,\rho_0). 
 \]
 We are only interested in $(x_0,t_0,\xi_0, \omega_0) \in L^\pm \R^4$ which shall be required. 
Moreover, we assume that the last coordinate of $\mathcal{H}^{-1}$ equals $\omega \mp |\xi|_g$, i.e, $ \mathcal{H}_* \rho (x,t,\xi,\omega) = \omega \mp  |\xi|_g$ (see Darboux' theorem \cite[Theorem 3.5.6]{Duistermaat}). 
The graph of $\mathcal{H}$ is denoted by $G_{\mathcal{H}}$. The graph of $\mathcal{H}^{-1}$, denoted by $G_{\mathcal{H}}^{-1}$, is obtained by changing the order of coordinates in $G_{\mathcal{H}}$. 

In the next step we follow Egorov's theorem \cite[Theorem 10.1]{Grigis-Sjostrand}. We fix Fourier integral operators $\mathcal{A} \in I^0_{cl}( \R^4 ,\R^4 ; G_{\mathcal{H}})$ 
 and $\mathcal{B} \in I^0_{cl}( \R^4 ,\R^4 ; G_{\mathcal{H}}^{-1})$ with  locally reciprocal  principal symbols (positively homogeneous of degree $0$) 
  such that 
$(x_0,t_0,\xi_0,\omega_0,y_0,s_0,\eta_0,\rho_0)$ is noncharacteristic for $\mathcal{A}$, 
$(y_0,s_0,\eta_0,\rho_0,x_0,t_0,\xi_0,\omega_0)$ is noncharacteristic for $\mathcal{B}$ and there is a conic neighbourhood $\mathcal{V}_1$ of $(y_0,s_0,\eta_0,\rho_0)$ and a conic neighbourhood $\mathcal{V}_2$ of $(x_0,t_0,\xi_0,\omega_0)$ such that 
 \begin{align}
 & WF(  \mathcal{B}\mathcal{A}  - id ) \cap \mathcal{V}_1 = \emptyset ,  \quad
WF(  \mathcal{A}\mathcal{B}  - id )  \cap \mathcal{V}_2 = \emptyset .
 \end{align}
 We may assume that the Schwartz kernels of $\mathcal{A}$ and $\mathcal{B}$ 
 are compactly supported distributions. 
The principal symbol of $ \mathcal{B}  (D_t \mp |D_x|_g)  \mathcal{A}   $ equals the coordinate $\rho$ in a conic neighbourhood of $(y_0,s_0,\eta_0,\rho_0)$. Hence, microlocally near $(y_0,s_0,\eta_0,\rho_0)$, 
\[
 \mathcal{B}  (D_t  \mp W) \mathcal{A}  \equiv   \mathcal{B}  (D_t  \mp |D_x|_g  +h) \mathcal{A}  
 \equiv D_s  + q, 
 \]
 modulo $\mps{-\infty}$, where $W$ is as in Lemma \ref{equdfkss} and $q \in \mpsc{0}$. 
 
\begin{definition}
Let $ S^m_{hom} (\R^4  \times \R^4 )$ stand for the symbols $a(x,t,\xi,\omega)$ in $S^m ( \R^4  \times \R^4  )$ that outside a bounded neighbourhood of the origin are positively homogeneous of degree $m$  in $(\xi,\omega)$. 
\end{definition}

The following lemma provides a way to get rid of the extra term $q$ 
(cf. \cite[Proposition 6.1.4]{FIO2}, \cite[Lemma 10.3]{Grigis-Sjostrand}). 

\begin{lemma}\label{toidssneho}
Let $q   \in \mpsc{0}$. There exists $A  \in \mpsc{0}$ 
such that 
\[
( D_s   + q )  A  - A  D_s   \in \mps{-\infty}.
\]
Moreover, there is a neighbourhood $U$ of $(y_0,s_0)$ in $\R^4$ and $a_0 \in M_3( S^0_{hom} ( \R^4  \times \R^4 ) )$ that is invertible as a matrix at every point $(y,s,\eta,\rho) \in U \times \R^4$ and satisfies 
\[
 A  -  a_0(y,s,D_y,D_s)   \in  \mpsc{-1},
\]
\end{lemma}

\begin{proof}[Proof of Lemma \ref{toidssneho}]
Write each entry $q^j_k$ of $q$ in the asymptotic form
\[
q^j_k(y,s,\eta,\rho) \sim   \sum_{m=0}^\infty q^k_{j,-m}(y,s,\eta,\rho), 
\] 
where $ q^k_{j,-m}(y,s,\eta,\rho)$ is positively homogeneous of degree $-m$ as a function of $\eta,\rho$. 
We substitute the ansatz $A^k_j = a^k_j(y,s,D_y,D_s)$, 
\[
a^k_j (y,s,\eta,\rho) \sim \sum_{m=0}^\infty a^k_{j,-m}(y,s,\eta,\rho) , \quad j,k=1,2,3, 
\] 
in 
\begin{equation}\label{ei2oa}
( D_s + q )  A   - A  D_s  \in M_3( \Psi^{-\infty} ( \R^4) ) ,
\end{equation}
apply the formula \eqref{gener}, 
and derive sufficient conditions by putting together the terms that are of the same degree. 
The equation \eqref{ei2oa} can be written as 
\begin{equation}\label{dfijos}
[D_s  , A^j_k ]  + q^l_k A^j_l  \in \Psi^{-\infty} ( \R^4), \quad j,k=1,2,3.
\end{equation}
For every $m=1,2,3,\dots$, let $a_{-m}$ and $q_{-m}$ be the matrices with entries $a_{j,-m}^k$ and $q_{j,-m}^k$, respectively. 
Setting the leading terms in the symbol of \eqref{dfijos} equal zero implies
\[
-i \{  \rho , a_{j,0}^k \} + q^l_{j,0} a^k_{l,0}= 0,
\]
that is,
\[
 \partial_s a_{0}  = -i q_{0}   a_{0}  .
\]
Instead of trying to solve the equation explicitly we refer to the basic theory of ordinary differential equations for existence of solutions. Provided an initial value at $s=s_0$, both uniqueness and existence of $s \mapsto a_0(y,s,\eta,\rho)$ on a closed interval follows from \cite[Theorem 5.1]{MR0273082}. 
We fix the initial value $a_0(y,s_0,\eta,\rho) = I$. 
Positive homogeneity of $q_0$  in $\rho,\eta$ then imply that the solution really is positively homogeneous of degree $0$ in $\rho,\eta$. 
Moreover, regularity of the solution in $(y,s,\eta,\rho)$ is a consequence of  \cite[Theorem 10.5]{MR0273082}. 
As the initial value is an invertible matrix and the solution is positively homogeneous of degree $0$, it follows by continuity that there is an open neighbourhood $U$ of $(y_0,s_0)$ in $\R^4$ such that $a_0(s,y,\eta,\rho)$ is invertible matrix for every $(y,s,\eta,\rho) \in T^*U \setminus \{0\}$. 
For the rest of the terms one derives
\begin{equation}\label{asdq2w}
\partial_s a_{-m} = -i q_{0} a_{-m} -i  f_{-m}   , \quad m=1,2,3,\dots
\end{equation}
where $ f_{-m}  $ is a matrix depending on entries of $\partial_{\eta,\rho}^\alpha q_{-w} D_{y,s}^\alpha a_{-v} $, where $v=0,\dots,m-1$ and other indices satisfy $|\alpha| + m = v+w$. 
Each $a_{-m} $ can be solved recursively from these equations. Proceeding inductively, similar arguments as in the case $m=0$ apply for existence and regularity of solutions with smooth initial values. 
 It is also straightforward to check that positive homogeneity of degree $-m$ in $\eta,\rho$  is consistent with the solutions. 

\end{proof}

 The following lemma is the construction of a microlocal parametrix. 
The proof (see Appendix \ref{appe1}) does not differ significantly from the scalar case (see \cite{FIO2}).
\begin{lemma}\label{wiikki}
For $L \in \mpsc{m}$, $m\in \R$, let $a_m \in M_3( S^{m}_{hom} ( \R^4 \times \R^4) )$ be the principal symbol of $L$ in the sense that 
\[
L - a_m (x,t,D_x,D_t) \in  \mpsc{m-1}. 
\]
 %
 %
 %
 %
 Assume that $a_m(x,t,\xi,\omega)$ is an invertible matrix at some $ (x_0,t_0,\xi_0,\omega_0) \in T^*\R^4$. 
  Then, there is $R \in  \mpsc{-m}$ and a conic neighbourhood $\mathcal{V}$ of $(x_0,t_0,\xi_0,\omega_0) $ in $T^*\R^4$ satisfying the conditions (i) and (ii) below. 
\begin{enumerate}
 \item[(i)] $R$ is a (two sided) microlocal inverse of $L$ in $\mathcal{V}$:
\[ 
WF (L^l_jR^k_l - \delta^k_j)   \cap \mathcal{V}  = \emptyset, \quad \text{and} \quad 
WF  (R^l_jL^k_l  - \delta^k_j)   \cap \mathcal{V} = \emptyset.
\]
\item[(ii)] The principal symbols of $R$ and $L$ are reciprocal in $\mathcal{V}$: There is $b_{-m} \in M_3( S^{-m}_{hom} ( \R^4 \times \R^4  ))$ satisfying
\[
R -  b_{-m}(x,t,D_x,D_t) \in \mpsc{-m-1}
\]
and 
 \[
 b_{-m}(x,t,\xi,\omega) = (a_m)^{-1} (x,t,\xi,\omega) , \quad \forall (x,t,\xi,\omega) \in \mathcal{V}. 
 \]
\end{enumerate}
 \end{lemma}
 Applying microlocality\footnote{$WF(Au) \subset WF(u)$ for $ A\in \Psi^m(\R^n)$} to Lemma \ref{wiikki} yields the following:
 \begin{corollary}\label{kosda}
Let $L$, $R$ and $\mathcal{V}$ be as in Lemma \ref{wiikki}. Then, 
\[
\mathcal{V} \cap \bigcup_{j=1,2,3} WF(L_j^ku_k) =\mathcal{V} \cap \bigcup_{j=1,2,3} WF(u_j) = \mathcal{V} \cap \bigcup_{j=1,2,3} WF(R_j^ku_k).
\]
for $u_k \in \mathcal{E}'(\R^4)$, $k=1,2,3$.
 \end{corollary}

 Let $A$ be as in Lemma \ref{toidssneho} and let $B$ be the corresponding microlocal inverse of it, as in Lemma \ref{wiikki}. 
Define $\mathcal{X} := \mathcal{A}  A$ and $\mathcal{Y} := B \mathcal{B} $. 
Applying the standard FIO calculus to $A^k_j,B^k_j \in \Psi_{cl}^0( \R^4) $, $\mathcal{A} \in I^0_{cl} ( \R^4,\R^4; G_{\mathcal{H}} )$, and $\mathcal{B} \in I^0_{cl} ( \R^4,\R^4; G_{\mathcal{H}}^{-1} )$ yields
\begin{align}
\mathcal{X}^k_j &= \mathcal{A}  A^k_j  \in I^{0+0}_{cl} ( \R^4 , \R^4; G_{\mathcal{H}} \circ  \text{diag}(T^*\R^4 \setminus \{0 \})  ) =  I^{0}_{cl} ( \R^4 , \R^4; G_{\mathcal{H}} ), \\
\mathcal{Y}^k_j &=   B^k_j \mathcal{B}  \in I^{0+0}_{cl} ( \R^4 , \R^4; \text{diag}(T^*\R^4 \setminus \{0\}) \circ G_{\mathcal{H}}^{-1} ) =  I^{0}_{cl} ( \R^4 , \R^4; G_{\mathcal{H}}^{-1} ),
\end{align}
together with the following identities for  $(y,s, \eta, \rho) $ in the domain of $\mathcal{H}$:
\begin{equation}\label{symbocalcu}
\begin{split}
 \sigma (\mathcal{X}^k_j ) (\mathcal{H}(y,s, \eta, \rho) , y,s, \eta, \rho) &= \sigma( \mathcal{A}  ) ( \mathcal{H}(y,s, \eta, \rho), y,s, \eta, \rho)  \sigma(A^k_j ) ( \text{diag} (y,s, \eta, \rho )), \\ 
  \sigma (\mathcal{Y}^k_j ) ( y,s, \eta, \rho, \mathcal{H}(y,s, \eta, \rho))&=  \sigma(B^k_j ) ( \text{diag}\mathcal{H}(y,s, \eta, \rho))  \sigma( \mathcal{B}  ) ( y,s, \eta, \rho, \mathcal{H}(y,s, \eta, \rho)).
  \end{split}
\end{equation}
Here $\sigma= \sigma_0 \in S^0 /S^{-1}$ refers to the principal symbol. 
As $\mathcal{X}\mathcal{Y} \equiv I $ and $\mathcal{Y}\mathcal{X} \equiv I $ near $(x_0,t_0,\xi_0,\omega_0)$ and $(y_0,s_0,\eta_0,\rho_0)$, respectively, there are conic neighbourhoods $\mathcal{V}$ and $ \mathcal{W} = \mathcal{H}^{-1} \mathcal{V}$ of $(x_0,t_0,\xi_0,\omega_0)$ and $(y_0,s_0,\eta_0,\rho_0)$, respectively, such that  
\begin{align}
&\mathcal{W} \cap   \bigcup_{j=1,2,3 } \mathcal{H}^{-1}  WF(f_j )    =  \mathcal{W} \cap   \bigcup_{j=1,2,3 } WF( \mathcal{Y}^k_j f_k ),  \label{meniscus_medialis}  \\ 
& \mathcal{V} \cap  \bigcup_{j=1,2,3 } \mathcal{H}  WF( f_j )  =  \mathcal{V} \cap  \bigcup_{j=1,2,3 } WF( \mathcal{X}^k_j f_k ). \label{meniscus_lateralis}
\end{align} 
for any $f_j\in \mathcal{D}' ( \R^4) $, $j=1,2,3$. 
See derivation of \eqref{meniscus_medialis} in Appendix \ref{appe1}. The other equation is computed similarly. 

\begin{definition}\label{deffa}
For $(x,t,\xi) \in \R^4 \times \R^3$ we define a smooth curve 
\[
\Sigma = \Sigma_\pm : \R \rightarrow T^*\R^4, \quad r\mapsto \Sigma (r) = (X(r),T(r),\Xi(r), \Omega (r)) \in L^\pm \R^4 , 
\]
 by 
 \begin{equation}\label{optomet}
X (r) = \gamma_{x,v}\left( \frac{r}{|\xi |_g} \right)  , \quad T (r) =t + r, \quad \Xi (r) =  \mp \flat_g \dot\gamma_{x,v}\left( \frac{r}{|\xi|_g} \right)    , \quad \Omega (r) = \pm |\xi|_g , 
\end{equation}
where $v :=  \mp \sharp_g \xi =  \mp g^{jk} \xi_k \partial_j$. 
\end{definition}

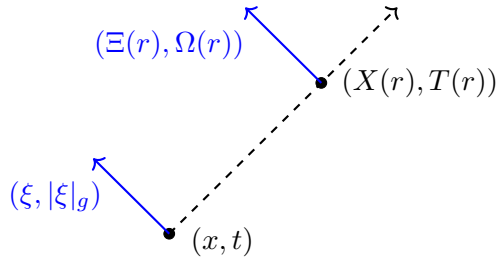
\begin{figure}[h]
\begin{tikzpicture}
\filldraw  (0,0) circle (2pt);
\filldraw  (2,2) circle (2pt);
\node[black] at (0.7,-0.1) {$(x,t)$};
\node[blue] at (-1.5,0.5) {$(\xi,|\xi|_g)$};
\node[black] at (3.3,2) {$( X(r), T(r) )$};
\node[blue] at (0,2.5) {$(\Xi(r), \Omega(r) )$};
\draw[->, dashed, thick](0,0) -- (3,3);
\draw[->,thick,blue](2,2) -- (1,3);
\draw[->,thick,blue](0,0) -- (-1,1);
\end{tikzpicture}
\caption{The curve $(X,T,\Xi,\Omega)$ translates the initial covector $(x,t,\xi , |\xi|_g)$ along the geodesic $(X,T)$ which is represented by the dashed line in the drawing.
The covector $(\Xi,\Omega)$ is normal to the geodesic $(X,T)$. }
\end{figure}

The lemmas below follow from standard geometric observations. For proofs, see Appendix \ref{appe1}.
\begin{lemma}\label{ds89askqeiorw}
The curve \eqref{optomet} is the unique bicharacteristic for the operator $D_t \mp |D_x|_g \in \Psi^1_{cl} ( \R^4)$ (i.e. an integral curve of the Hamilton vector field $H_{p^\pm}$ for $p^\pm := \omega \mp |\xi|_g$ in $\text{char}(p^\pm)= L^\pm \R^4$) with $(x,t,\xi,  \pm |\xi|_g)$ as the initial value at $r=0$. 
\end{lemma}

\begin{lemma}\label{lahjboelihjs125}
Let $(y,s,\eta) \in \R^4 \times \R^3$ and $r_0>0$ such that the segment $\{ (y,s + r,\eta,0) \in T^* \R^4 : r\in (-r_0,r_0) \}$  lies in the domain of the canonical transormation $\mathcal{H}$. Then, the curve
\[
 r \mapsto \mathcal{H} (y,s+r,\eta,0) , \quad  r\in (-r_0,r_0)
\]
equals  $r \mapsto \Sigma(r)$, $r \in (-r_0,r_0)$ with $(x,t,\xi,\pm |\xi|_g) = \mathcal{H} (y,s,\eta,0) $ as the initial value. That is; the transformation $\mathcal{H}$ takes a bicharacteristic for $D_s \in \Psi^1_{cl} ( \R^4)$ into a bicharacteristic for $D_t \mp  |D_x|_g  \in \Psi^1_{cl} ( \R^4)$. 
\end{lemma}

\begin{definition}
We define 
\begin{equation}\label{jksks20kasfdbg}
\Lambda_{f,\pm} := \Big\{ \big(   \Sigma(r)  \ ; \ x,t,\xi, \pm | \xi|_g \big) :   r \in [0,\infty), \ (x,t,\xi) \in \R^4 \times \R^3 \Big\} ,
\end{equation}
where $ \Sigma(r) = \Sigma_\pm (r)$ is the curve \eqref{optomet} with the initial value $\Sigma(0) =  (x,t,\xi, \pm | \xi|_g) \in L^\pm  \R^4$. 
This is the forward propagating part of the characteristic flow-out canonical relation associated with $D_t \mp |D_x|_g$. There is also the backwards propagating part $\Lambda_{b,\pm}$, defined as 
\[
\Lambda_{b,\pm} := \Big\{ \big(   \Sigma(r)  \ ; \ x,t,\xi, \pm | \xi|_g \big) :   r \in (-\infty,0], \ (x,t,\xi) \in \R^4 \times \R^3 \Big\}. 
\]
The complete flow-out canonical relation associated with $D_t \mp |D_x|_g$ is the union 
\[
\Lambda_\pm = \Lambda_{f,\pm} \cup \Lambda_{b,\pm} . 
\]
\end{definition}

\begin{definition}
Let $\pi_{\R^3} : T^*\R^4 \rightarrow \R^3$ be the projection $\pi_{\R^3}(x,t , \xi, \omega ) : = x$.  For $(z,\theta,\beta) \in \R^3 \times \S^2 \times (0,1)$ we define 
\[
\begin{split}
\mathcal{B} K  (z,\theta,\beta)   := \{ x \in \R^3 :   |\theta|_g (x ) = \beta^{-1} , \ x =z+  r \theta, \ r \in \R \}  .
\end{split}
\]
\end{definition}
The set $\mathcal{B} K  (z,\theta,\beta)$ consists of all the spatial points where a particle that moves through the point $z$ at a constant velocity $\beta \theta$ breaks the ``light-barrier'', that is, moves exactly at the speed of waves in the medium. 
The following proposition is the main tool in this paper.

\begin{proposition}\label{pinheiro}
Fix $z \in \R^3$, $\beta \in (0,1) $, and $\theta \in \S^2$. Let $g$ be a Riemannian metric  in $\R^3$   with $|\cdot|_g\geq 1$ on $\S^2$ 
and assume that  
\begin{align}\label{asjd9w1sadadsrertew}
& g( \nabla_\theta \theta, \theta) (x) \neq 0, \quad \forall x   \in \mathcal{B} K  (z,\theta,\beta), 
 \end{align} 
 $($i.e. $\nabla_\theta |\theta|_g  (x) \neq 0$,  $\forall x   \in \mathcal{B} K  (z,\theta,\beta))$. 
%
 Let 
\[
S_j \in I^m_{cl}( \R^4;N^*K(z,\theta, \beta) ),  \quad j=1,2,3,
\]
be compactly supported 
and $P \in \mpsc{2}$ be of the form
\[
P   -  (  \partial_t^2 - g^{jk} \partial_j \partial_k  )I -  F  \in \mps{-\infty}, 
\]
for some stationary operator $F =[  F^k_j(x,D_x) ]_{j,k=1,2,3} $ of order $1$. 
Consider distributions 
$u_j  \in \mathcal{D}'(\R^4)$, $j=1,2,3$, that for large $T>0 $ obey
\begin{align}
&P^k_j u_k - S_j  \in C^\infty ( \R^4),  \label{luosysteemi1}\\
&u_j|_{ t \leq -T} \in  C^\infty ( \R^3 \times (-\infty, -T]  ).   
\end{align}
Then, 
\[
\bigcup_{j=1,2,3} WF(u_j) \setminus \bigg( \bigcup_{k=1,2,3} WF(S_k)   \bigg)  \subset \bigcup_{j=1,2,3} \Lambda_f \circ WF(S_j) , \quad \Lambda_f:= \Lambda_{f,+} \cup \Lambda_{f,-} . 
\]
Moreover, if the initial value $(x,t,\xi, \pm |\xi|_g) = \Sigma(0) \in L^\pm \R^4$ is the first intersection between a bicharacteristic $\Sigma(r)$ of the form \eqref{optomet} and $ \bigcup_{j=1,2,3}  WF(S_j)$, and the positively homogeneous principal symbol (of degree $m-1/2$) for at least one of the components $S_{j}$, $j=1,2,3$ does not vanish at $(x,t,\xi, \pm |\xi|_g)$ (i.e. the 3-vector of these principal symbols is nonzero), then the open segment between the first and the second intersection lies in the wave front set. That is;
\[
\Sigma(r) \in  \bigcup_{j=1,2,3} WF(u_j) 
 \]
 for 
 \[
0 < r < \sup \Big\{  s>0 : \Sigma(s) \notin  \bigcup_{k=1,2,3} WF(S_k) \Big\}.
\]
\end{proposition}

For proving the proposition we need the following lemma: 

\begin{lemma}\label{rosl}
Assume that the Riemannian metric $g$ on $\R^3$ with $|\cdot |_g\geq 1$ satisfies 
\[
g( \nabla_\theta \theta ,\theta)(x)   \neq 0 , \quad \forall x  \in \mathcal{B}K(z,\theta,\beta). 
\] 
Then the bundle $L^\pm \R^4 $ is transversal to $  N^*K(z,\theta,\beta)$ and the intersection of $N^*K(z,\theta,\beta)$ with a curve $\Sigma(r)$ of the form \eqref{optomet} is a discrete set. 
\end{lemma} 
\begin{proof}[Proof of Lemma \ref{rosl}]
Denote $K := K(z,\theta,\beta)$. 
 We begin by proving the first claim. 
 As 
\[
\begin{split}
&\text{dim} (   T_v   N^* K  + T_v  L^\pm \R^4 ) 
\\
&=  \text{dim}( N^* K )+  \text{dim} ( L^\pm \R^4 ) -  \text{dim} ( T_v N^* K  \cap  T_v L^\pm \R^4 ) \\ 
&= 4 + 7-  \text{dim} ( T_v N^* K  \cap  T_v L^\pm \R^4 )  \\
&= \text{dim} (T_v T^* \R^4 ) + 3-  \text{dim} ( T_v N^* K  \cap  T_v L^\pm \R^4 ) ,   \quad v= \mathcal{H} w \\ 
%
\end{split}
\] 
we need to show that $ \text{dim} ( T_v N^* K  \cap  T_v L^\pm \R^4 )  \leq 3$.
Recall that  
\[
N^*K = \{ (z+ t\beta \theta, t \ ; \  \xi ,  - \beta \xi \cdot \theta) : t\in \R, \ \xi \in \R^3 \setminus \{0\} \} . 
\]
Thus, for $v\in N^*K$ we have the expression
 \[
T_v N^*K = \{ \big( ( \delta t) \beta \theta, \delta t  ,   \delta \xi ,  - \beta ( \delta \xi ) \cdot \theta \big) : \delta t\in \R, \ \delta \xi \in \R^3 \} . 
 \]
As $L^\pm\R^4 = \{ (x,t,\xi,\pm |\xi|_g ) : (x,t,\xi) \in \R^4 \times (\R^3\setminus \{0\}) \}$ we have for $v = (x,t,\xi, \pm |\xi|_g) \in L^\pm \R^4$ that
\[
T_v L^\pm\R^4 = \Big\{  ( \delta x , \delta t, \delta \xi ,   \delta \omega) : \delta \omega = \pm \frac{g(\xi,\delta \xi) +    \frac{1}{2}  \langle d_x g ( \xi,\xi)   , \delta x \rangle }{| \xi |_g(x)} \Big\}.
\] 
Consequently, for $v= ( x(t)  ,t,\xi, \omega(t,\xi) ) \in N^*K \cap L^\pm \R^4$ ($x(t) =z+ t \beta\theta$, $\omega (t,\xi) = - \beta \xi \cdot \theta=  \pm |\xi|_g(z+  t \beta \theta) $)   we have 
\[
\text{dim} ( T_v N^*K  \cap T_v L^+\R^4)    =   \text{dim} ( \ker L ) .
\] 
where  $ L = L_{z,\xi,t,\beta,\theta} : \R^4 \mapsto \R$ is a linear map defined by
\[
L :  ( \delta \xi, \delta t)  \mapsto \pm   \frac{g(\xi,\delta \xi)+     \frac{\beta}{2}   \langle d_x g( \xi,\xi) , \theta  \rangle \delta t }{| \xi |_g } \Big|_{x=z+ t\beta \theta } +\beta ( \delta \xi ) \cdot \theta .
\]
To prove that the dimension of the kernel is less or equal to $3$ we need to show that the map $L$ is not identically zero. By setting $L |_{\R^3 \times \{0\}}= 0$ one obtains 
\begin{equation}\label{apo}
 \pm \frac{  \xi     }{| \xi |_g    }    = -\beta \flat_g \theta  \quad \text{at} \quad x=z+  t\beta \theta
 \end{equation}
  which is possible only if  $| \theta |_g  = \frac{1}{\beta}$ at $x= z+  t\beta \theta$. 
  Let us  study $L$ at such points. 
   By substituting \eqref{apo} in $L (0  , \beta^{-3} )$ implies 
   \[
    L (0, \beta^{-3}) = \pm   \frac{1}{2}   \langle d_x g(  \theta ,  \theta ) , \theta  \rangle|_{x=z+ t\beta \theta }  =   \pm g( \nabla_\theta \theta, \theta)|_{x=z+ t\beta \theta } 
    \] 
    which is nonzero by  assumptions. In conclusion, the dimension of $\text{ker}L$ is at most $3$ which finishes the first part of the proof. 

Let us now show that the intersection of $N^*K$ and the curve $\Sigma(r) = (X(r),T(r),\Xi(r),\Omega(r))$ of the form \eqref{optomet} is a collection of discrete points. To prove this we set $\Sigma(0) \in N^*K$ and show that $(X(r),T(r))$ is not in $K$ for $r\neq 0$ in a small neighbourhood of $0$. By definition, $X(r) = \gamma_{v}(r/ |\xi|_g)$, $v= \mp \sharp_g \xi$. 
If $( X(r) ,T(r) ) $ is not tangent to $K$ at $r=0$, then the claim clearly holds. Thus, we may assume that $(\dot{X}(0),\dot{T}(0))  \in TK$ which implies the following approximation near $r=0$: 
\[
X(r) =X(0) + \beta  \theta  r - h(r),
\]
 where $h(r) = O(r^2)$. In addition, $(\dot{X}(0),\dot{T}(0))  \in L^\pm \R^4$ so at $X(0)$ we must have
 $
1 = | \dot{T} (0) | =   | \dot{X} (0) |_g  =  \beta | \theta|_g,
  $ 
  that is, $ | \theta|_g= \frac{1}{\beta}$. In particular, $g(  \nabla_{ \theta}     \theta, \theta) |_{x=X(0) } \neq 0$ by assumptions.  As $r \mapsto X(r)$ is a geodesic, it satisfies $ \nabla_{ r } \dot{X}  = 0$, where $\nabla_r  V =   \dot{V}^l  \partial_l  + \Gamma^l_{jk} V^j V^k \partial_l $ is the covariant derivative of a vector field along a curve (see e.g. \cite{MR1468735}). On the other hand, 
  \[
  \begin{split}
  g(   \nabla_r \dot{X} , \dot{X} ) |_{r=0} =  \beta^2 g(  \nabla_{ \theta}     \theta , \theta)  |_{x=X(0) }  -  \beta  g(  \nabla_r  \dot{h} ,   \theta)  |_{r= 0} \\
  \end{split} 
    \]
    Thus, $  g(  \nabla_r  \dot{h} ,   \theta)  |_{r= 0} =  \beta g(  \nabla_{ \theta}     \theta , \theta)  |_{x=X(0) }  \neq 0$ which is possible only if  $ h(r) =  \mu r^2 + O(r^3)$ around $r=0$  for some nonzero vector $ \mu $ (i.e. the curve $X(r)$ ``bends'' at $r=0$).  In conclusion, for $r\neq 0$ near the origin we have that $(X(r),T(r)) $ does not lie in the line $K$.
\end{proof}

\begin{proof}[Proof of Proposition \ref{pinheiro}]
 We relax the notation by omitting the parameters $(z,\theta,\beta)$ whenever there is no danger of confusion. 
 The proof follows the standard scheme of reducing the original vector-valued system microlocally along the characteristic flow into independent scalar transport equations. 
 Solving such equations is a simple task and requires only extending the fundamental theorem of calculus. 
 
For a fixed sign $\pm \in \{+,-\}$ and a vector $(x_0,t_0,\xi_0,\omega_0) \in  T^*\R^4   $, $\omega_0 \neq \mp |\xi_0|_g$ we express the wave  $u = u_j dx^j$ microlocally in a suitable conic neighbourhood $\mathcal{V}= \mathcal{V} _{x_0,t_0,\xi_0,\omega_0}$ of $(x_0,t_0,\xi_0,\omega_0)$ as a sum 
\begin{equation}\label{slqwegg}
u \equiv u_\pm + u_0   \mod C^\infty , 
\end{equation}
of a term $u_\pm= u_{\pm, \mathcal{V}}$ which creates (or annihilates) singularities that propagate forwards in time along the characteristic flowout from the source, and a residual term $u_0 = u_{0, \mathcal{V}}$ with a flow-invariant wave front set. That is; 
\[
u_{\pm,j} \in I_{cl}^r ( \R^4;  \Lambda_{f,\pm} \circ   N^*K   ), \quad j=1,2,3, \quad \text{microlocally away from }N^*K, 
\]
which can be chosen such that $u_{\pm} \equiv 0$ if $(x_0,t_0,\xi_0,\omega_0) \notin \bigcup_{j=1,2,3} WF(S_j)$, 
 and 
 \[
  \Lambda_{\pm}  \circ WF_{\mathcal{V} }(u_{0}) = WF_{\mathcal{V} }(u_{0})  
  \]
  where
  \[
  WF_{\mathcal{V} }(f) :=\bigcup_{j=1,2,3} WF(f_j) \cap \mathcal{V} , 
  \]
  all in a suitable conic neighbourhood $\mathcal{V} $ of $(x_0,t_0,\xi_0,\omega_0)$.
%
The residual term just moves existing singularities along bicharacteristics \eqref{optomet} in the sense that the wave front set $\bigcup_{j=1,2,3} WF(u_{0,j})$ in one region $\mathcal{V}' \subset \mathcal{V}$ defines it in the flowout $ \Lambda_\pm \circ \mathcal{V}' $ within $\mathcal{V}$. 
The wave front set of $u$ away from $ L^\mp \R^4 $ is computed by solving the microlocal expressions $u=u_\pm + u_0$ along curves of the form \eqref{optomet}. 
Let us briefly explain this before constructing the terms. First of all, early points in curves \eqref{optomet} have not yet hit the set $ \R^3 \times [-T,\infty) $  so 
regularity of the wave $u$ outside $\R^3 \times  (-T,\infty)$ implies that no singularities along the curve occur in $u$ until intersection with $\bigcup_{j=1,2,3}  WF(S_j)$ the first time in $\R^3 \times  (-T,\infty)$.   
Notice that the condition $| \theta |_g \geq 1$ on $\S^2$  ensures that each bicharacteristic really goes through the Cauchy surface $t= -T$. 
We must therefore have $u\equiv u_\pm $ at the first intersection of the curve and $\bigcup_{j=1,2,3}  WF(S_j)$. Hence, a forward propagating singularity is created at the point if for some $j$ we can show that the positively homogeneous principal symbol of $u_{\pm,j}$ does not vanish there. 
On the other hand, in a suitable conic neighbourhood around any point outside $\bigcup_{j=1,2,3}  WF(S_j)$ the expression $u=u_\pm + u_0$  leads to smooth $u_\pm$ and hence flow-invariant $\bigcup_{j=1,2,3}  WF(u_j)$.  
Thus, in a convenient sequence of conical neighbourhoods that cover a given bicharacteristic segment \eqref{optomet} emanating from the source each created singularity gets transported arbitrarily far until it possibly gets annihilated at the source. 


Let us now derive the microlocal expression above. 
By Lemma \ref{wiikki} it suffices to study $u$ near $(x_0,t_0,\xi_0,\omega_0) $, $\omega_0= \pm |\xi_0|_g$.  
As before, we let $\mathcal{H}$ be the homogeneous local transformation  \eqref{fixh} on a conic neighbourhood of the point with $ \omega \mp |\xi|_g$ being the last coordinate. 
Applying  Lemma \ref{equdfkss} and Lemma \ref{toidssneho} we derive  
\[
\begin{split}
  \mathcal{Y}  S  =  \mathcal{Y}  P u  \equiv \mathcal{Y}  (D_t I  - W) Z u  \equiv \mathcal{Y} (D_t I  - W)\mathcal{X}\mathcal{Y}  Z  u \equiv  B \mathcal{B}  (D_tI  - W) \mathcal{A} A \mathcal{Y}  Z u  \\
  \equiv  B (D_tI + q )A \mathcal{Y} Z u  \equiv  BA D_s \mathcal{Y} Z u  \equiv  D_s \mathcal{Y}Z u  \mod C^{\infty}   
 \end{split} 
\]
microlocally near $(y_0,s_0,\eta_0,\rho_0)$. 
Thus, there are $\delta,\epsilon>0$ such that in the open cylinder $\cyl := B(y_0,\delta ) \times (s_0- \epsilon,s_0 + \epsilon)$ we have modulo $C^\infty( \cyl )$ the local transport equations 
\begin{equation}\label{dif1}
D_s \tilde{u}_j (y,s)  \equiv \tilde{S}_j  (y,s) ,  \quad  j=1,2,3, \quad  (y,s) \in  \cyl , 
\end{equation}
for 
\begin{align}
\tilde{u}_j :=&\varrho  \mathcal{Y}^k_jZ^l_k u_l \in \mathcal{E}'(\R^4),
 \\
 \tilde{S}_j := & \varrho \mathcal{Y}^k_j  S_k \in I^{m}_{cl} (  \mathcal{H}^{-1} N^* K  ; \R^4 ),
  \\
 \varrho:=&  \chi\left( \bigg| \frac{(D_y,D_s)}{|(D_y,D_s)|}- \frac{(\eta_0,\rho_0)}{|(\eta_0,\rho_0)|} \bigg| \right) \in \Psi^0_{cl} ( \R^4), 
\end{align}
where $\chi \in C_c^\infty ( \R)$ is a smooth cut-off function that vanishes outside $(-\delta,\delta)$ and equals 1 in $(-\delta/2,\delta/2)$. 
Microlocally near $(y_0,s_0,\eta_0,\rho_0)$  we have that $\tilde{u} \equiv \mathcal{Y} Z u$ and $\tilde{S} =  \mathcal{Y} S$. 
In view of (\ref{meniscus_medialis}-\ref{meniscus_lateralis}), Corollary \ref{kosda} (applied to $Z$) and Lemma \ref{lahjboelihjs125} it suffices to study singularities of each $\tilde{u}_j$ along bicharacteristics of $D_s$. 
The original wave $u$ is derived microlocally near $(x_0,t_0,\xi_0,\omega_0)$ by applying the inverse  of $\mathcal{Y} Z$ to $\tilde{u}$. 
Formally, 
\begin{equation}\label{asidj11}
u \equiv R \mathcal{X}  \tilde{u} \equiv  R \mathcal{X} E_f  \tilde{S}   + R \mathcal{X} v \mod C^\infty
\end{equation}
where $R$ is the microlocal inverse of $Z$ near $(x_0,t_0,\xi_0,\omega_0)$, $E_f$ stands for the forward fundamental solution $E_f (s,y ; \tilde{s}, \tilde{y}) =  i H(s- \tilde{s} ) \delta_0 ( y- \tilde{y}   ) $  (operating separately for each coordinate: $(E_f u )_j := E_f  u_j$), and $v \in \text{ker}(D_s)$ is a residual term with wave-front invariant in the characteristic flow of $D_s$:
\begin{equation}\label{239dskm}
\exists \Sigma_j \subset T^*B(y_0,\delta) \  : \  WF(v_j) =\{ (y,s,\eta,0): (y,\eta) \in \Sigma_j \}, \quad j=1,2,3,
\end{equation} 
(see \cite[Theorem 6.1.1]{FIO2}).\footnote{Choosing a different fundamental solution, say the backwards propagating solution $E_b (s,y ; \tilde{s}, \tilde{y}) =  -i H( \tilde{s}-s ) \delta_0 ( y- \tilde{y}   ) $ leads to a similar expression. 
In general, the difference $E_1  \tilde{S} - E_2  \tilde{S} $ for two solution operators $E_1,E_2$ solves $D_s u = 0$, and therefore is invariant along the bicharacteristics. }
The creation term is $u_\pm :=  R \mathcal{X} E_f \tilde{S} $, whereas the residual term is given by $u_0 := R \mathcal{X} v$. 
%
%
%
For proving that $E_f  \tilde{S}$, and hence the expression \eqref{asidj11} is well defined and compatible with \cite[Proposition 2.1]{Greenleaf-Uhlmann} it suffices to show that the characteristic set 
$\{ (y,s,\eta,0) : (y,s) \in \cyl \} $
of $D_s  :  \mathcal{D}' (  \cyl ) \rightarrow \mathcal{D}' (  \cyl ) $ is transversal to $  \mathcal{H}^{-1} N^*K $ and each bicharacteristic $\{ (y,s,\eta,0) : s \in (s_0-\epsilon, s_0 + \epsilon) \}$, $y\in B(y_0,\delta)$ intersects $\mathcal{H}^{-1} N^*K$ finite number of times. These properties follow for small $\epsilon,\delta >0$ by combining Lemma \ref{lahjboelihjs125} with Lemma \ref{rosl} and hence the referred proposition applies and yields the expression for $u_\pm$ as a vector of Lagrangian distributions associated with the pair $(N^*K, \Lambda_{f,\pm} \circ N^*K)$. In particular, microlocally away from $N^*K$ the term $u_\pm $ is a Lagrangian distribution over  $\Lambda_{f,\pm} \circ N^*K$. 
%

Finally we show that a propagating singularity in $u_\pm$ is created at a point where the principal symbol of $S$ is non-vanishing. As $\tilde{u} \equiv E_f \tilde{S} \equiv E_f \mathcal{Y} S $ microlocally near $(x_0,t_0,\xi_0,\omega_0) $, 
the standard principal symbol formula for compositions of FIOs (for $E_f$ the formula is provided in the proposition referred above)  implies that for $ (x_0,t_0 , \xi_0,\omega_0)  \in L^\pm \R^4 \cap N^*K$, sufficiently small $s>0$ and the positively homogeneous principal symbols $\tilde{p}_{\pm,j}$,  $\tilde{q}_j$, $q_j$ of $\tilde{u}_{\pm,j}$, $\tilde{S}_j$, $S_j$ we have 
\[
\tilde{p}_{\pm,j} (y_0,s_0 + s, \eta_0,\rho_0) \simeq \tilde{q}_j (y_0,s_0 , \eta_0,\rho_0)  = M^k_j  q_k (x_0,t_0 , \xi_0,\omega_0) , 
\]
where  $M$ is an invertible matrix given as the principal symbol of $   \mathcal{Y}$ at $ (y_0,s_0 , \eta_0,\rho_0; x_0,t_0 , \xi_0,\omega_0) $. This implies that $\tilde{u}_{\pm}$ creates at $(y_0,s_0 , \eta_0,\rho_0)$ a forwards propagating singularity provided that the principal symbol $q_k (x_0,t_0 , \xi_0,\omega_0)$ of $S_k$ does not vanish for at least one $k=1,2,3$. 
Applying this to $(x_0,t_0,\xi_0,\omega_0) = (x,t,\xi,|\xi|_g)$, where $ (x,t,\xi,|\xi|_g)$ is as in the assumptions, and using $R \mathcal{X}$ to transform $\tilde{u}_{\pm}$ microlocally back into $u_\pm$ yields the creation of singularity for $u_\pm$ at  $ (x,t,\xi,|\xi|_g)$, hence finishing the proof. 
%
%

\end{proof}

\section{Proof of Theorem \ref{maineee}}

In this section we derive the claim of Theorem \ref{maineee}. The proof is constructed in three stages. The first and perhaps the most difficult step is to prove Lemma \ref{fd5sol345} below. 
The second step is to derive Proposition \ref{234235rwerweght} from the lemma. The claim of Theorem \ref{maineee} is deduced in the final step after that.

\begin{lemma}\label{fd5sol345}
Assume that the conditions of Theorem \ref{maineee} hold. 
Fix $z \in U$ and let $x_0  \in \Upsilon$ be a \footnote{There can be many of such points.} nearest point to $z$ on the boundary with respect to $g_1$.  
Then there is a neighbourhood $\mathcal{Q} \subset  \Upsilon$ of $x_0$ in $\Upsilon$ such that 
\[
\text{dist}_{g_2}( z,x) \leq \text{dist}_{g_1}( z,x), \quad \forall x \in \mathcal{Q}. 
\]
Moreover, if $\gamma_1 : [0,l] \rightarrow \R^3$ with  $ l:= \text{dist}_{g_1}( z,x)$ is a geodesic segment in $(\R^3,g_1)$ from $\gamma_1(0) = z$ to $\gamma_1(l) = x$, then the unique geodesic $\gamma_2$ in $(\R^3,g_2)$ with $(\gamma_2(l) ,\dot\gamma_2(l) )=(\gamma_1(l) ,\dot\gamma_1(l) )$ satisfies $\gamma_2(0)=z$. 
\end{lemma}
\begin{proof}
Fix $z$ and $x_0$ as in the assumptions and 
 a small neighbourhood $\mathcal{Q} \subset \Upsilon$ of $x_0$ in $\Upsilon$.   
By making the neighbourhood small enough we may assume that every shortest geodesic segment in $(\R^3,g_1)$ from $z$ to a point in $\mathcal{Q}$ lies in $W$ and hits the boundary $\partial W$ transversally and for the first time at the endpoint. 
Fix arbitrary $x \in \mathcal{Q}$. 
Let $\gamma_1 : [0,l] \rightarrow \R^3$ be an unit speed geodesic in $(\R^3,g_1)$ from $z$ to $x$ such that $l= \text{dist}_{g_1}(z,x)$. 
Denote $v := \dot\gamma_1 ( 0 )$. Let $\gamma_2 $ be the unique geodesic in $(\R^3,g_2)$ such that $(\gamma_2(l) , \dot\gamma_2(l)) = ( \gamma_1(l),\dot\gamma_1(l))$. Since $g_1=g_2$ on $\Upsilon$, also the geodesic  $\gamma_2$ has unit speed. 
Define 
\[
M:= \{ (\theta,\beta) \in \S^2 \times \mathcal{I}_z : (z,0 ; -  \flat_{g_1} v ,  1  ) \in N^*K(z,\theta ,\beta)  \cap L^+(  \R^3 ; g_1) \},
\]
where $\mathcal{I}_z$ is as in the assumptions.  
 The space $M  $ is a trivial bundle over $ \mathcal{I}_z$,  consisting of distinct loops $M_\beta = M \cap (\S^2 \times \{ \beta \})$ on the surface $\S^2$ as fibers. 
In some coordinates the loops are just circles. 
 For each $\theta \in M_\beta$ there is the antipodal point $\tilde\theta = \tilde\theta(\theta) \in M_\beta $ in the circle $M_\beta$ which satisfies
\begin{equation}\label{sif}
N^*K(z,\theta ,\beta) \cap N^*K(z,\tilde\theta ,\beta) \cap L^+(  \R^3 ; g_1) = \{ (z,0 ; - h \flat_{g_1} v , h  ) :  h > 0 \}. 
\end{equation}
The construction  is illustrated in Figure \ref{fidur}. 

\begin{figure}[h]
\begin{tikzpicture}

\draw[dashed] (-3,0) ellipse (0.2 and 1);
\draw[rotate around={-37:(0,0)}, dashed] (-3,0) ellipse (0.2 and 1);
\draw[rotate around={-18.5:(0,0)}, color=red] (1.5,0) ellipse (0.1 and 0.5);

\draw[dashed]  (0,0) -- (-3,-1);
\draw[rotate around={-37:(0,0)}, dashed] (-3,1) -- (0,0) ;

\draw[->] (-1,0) -- (1.5,0);
\draw[->,rotate around={-37:(0,0)}] (-1,0) -- (1.5,0);

\node[label=  $\theta$] at (1,0) {};
\node[label=  $\tilde\theta$] at (0.7,-1.4) {};
\node[ color=blue, label=  ${\color{blue}-\flat_{g_1} v}$,] at (-1.5,0.5) {};
\draw[color=blue] (0,0) -- (-4.5,1.5);
\draw[thick, color=blue, ->] (0,0) -- (-1.5,0.5);

  \end{tikzpicture}
\caption{A projection of the construction in \eqref{sif} to the space $\R^3$. The loop $M_\beta$ is illustrated in red. The intersections $N^*K(z,\theta ,\beta) \cap  L^+(  \R^3 ; g_1)$ and $N^*K(z,\tilde\theta ,\beta) \cap  L^+(  \R^3 ; g_1)$ appear as two cones (the dashed lines) that touch each other at the half-open segment $\{ (z,0 ; - h \flat_{g_1} v , h  ) :  h > 0 \}$ (in blue).}\label{fidur}
\end{figure}
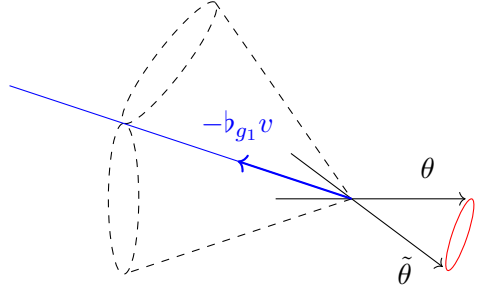

Let $\gamma_3$ be the unique unit-speed geodesic in $( \R^3,g_2)$ through $ x $ that is tangent to the mirror image of $ \dot\gamma_2(l)$ over the boundary $T_x\partial W$ at that point. 
Using the fact that $M$ is two dimensional as a manifold we fix $(\theta,\beta) \in M$ such that the lines $K(z,\theta,\beta)$ and $K(z,\tilde\theta,\beta)$ avoid the time-like geodesics 
\begin{equation}\label{123geo}
\{  (\gamma_{\tau} (r), r) : r \in \R, \ \tau=1,2,3 \}
\end{equation}
 inside $(B \times \R) \setminus \{(z,0)\}$ 
 for a large closed Euclidean ball $B \subset \R^3$ with radius $R>\sup_{\beta\in \mathcal{I}_z}R_\beta$. By assumptions, $B \times \R $ contains the singular supports of the sources. 
We would like to apply Proposition \ref{pinheiro} to systems associated with the parameters $(\theta,\beta)$ and $(\tilde\theta,\beta)$. However, it is possible that the rays $K(z,\theta,\beta)$ and $K(z,\tilde\theta,\beta)$ do not break the velocity barriers $\mathcal{B}K(z,\theta,\beta)$ and $\mathcal{B}K(z,\tilde\theta,\beta)$ in the required way. To get around this we perturb the velocity $\beta$ slightly and then apply Sard's theorem. 
 If one makes the parameter $\beta$ slightly larger while keeping everything else fixed, the rays still avoid the geodesics \eqref{123geo} in $(B \times \R) \setminus \{(z,0)\}$ but the intersection \eqref{sif} splits into two lines instead of one (See Figure \ref{giuel}). 
\begin{figure}[h]
\begin{tikzpicture}

\draw (-2.8,0) ellipse (0.2 and 1);
\draw[rotate around={-30:(0,0)}] (-2.8,0) ellipse (0.2 and 1);

\draw (-2.8,1) -- (0,0) -- (-2.8,-1);
\draw[rotate around={-30:(0,0)}] (-2.8,1) -- (0,0) -- (-2.8,-1);


\draw[color=red, ->] (0,0) --  (-2.65,0.72);
\draw[color=blue, ->] (0,0) --  (-2.9,0.86);
%
%
%
  \end{tikzpicture}
 \caption{Increasing $\beta$ makes the two touching cones in Figure \ref{fidur} wider. Intersection of the wider cones consists of two lines, illustrated here by the blue and red arrows.}\label{giuel}
\end{figure}
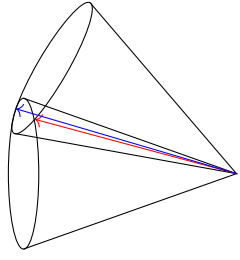
The splitting, however, happens continuously with respect to $\beta$. Thus, for every conic neighbourhood $\mathcal{V}_z \subset T_z^*\R^3$ of $ - \flat_{g_1}v$ in $T_z^*\R^3$ there is $\delta >0$ such that 
\[
N^*K(z,\theta ,\beta + \beta') \cap N^*K(z,\tilde\theta ,\beta+\beta') \cap L^+(  \R^3 ; g_1)   \subset \{ (z,0,\xi, |\xi|_g ) : \xi \in \mathcal{V}_z \} , \quad \forall \beta' \in [0,\delta) . 
\]
That is; for fixed $(\beta,\theta) \in M$, and associated $\tilde\theta$, we may add a small positive perturbation to $\beta$ without changing the intersection $\eqref{sif}$ significantly. 
By applying Sard's theorem to the functions $r \mapsto |\theta|_{g_\lambda}( z + r \theta)$ and $r \mapsto |\tilde\theta|_{g_\lambda}( z + r \tilde\theta)$ one constructs a sequence $\beta_j' \longrightarrow 0$ such that for every $j=1,2,3,\dots$ and $\lambda=1,2$ both of the following conditions hold:
\[
\begin{split}
\nabla_\theta |\theta|_{g_\lambda } (x) \neq 0, \quad \forall x \in \mathcal{B}K(z,\theta,\beta+\beta_j'), 
\\
\nabla_{\tilde\theta} |\tilde\theta|_{g_\lambda } (x) \neq 0, \quad \forall x \in  \mathcal{B}K(z,\tilde\theta,\beta+ \beta_j').
\end{split}
\]
This is equivalent to saying that $g_\lambda (\nabla_\theta \theta , \theta )|_x \neq 0$ and $g_\lambda (\nabla_{\tilde\theta} \tilde\theta , \tilde\theta )|_x \neq 0$ in the associated points 
 and hence the phase velocity barrier is broken in the required way for the approximating rays.  
 We may assume that $\beta+\beta_j'$ lies in $\mathcal{I}_z$ for all $j=1,2,3,\dots$ by excluding a finite number of terms in the beginning of the sequence.

 Let us fix $v_j \in T_z \R^3$, $j=1,2,3,\dots$, such that 
 \[
 (z,-\flat_{g_1}v_j,1) \in N^*K(z,\theta ,\beta + \beta_j') \cap N^*K(z,\tilde\theta ,\beta+\beta_j') \cap L^+(  \R^3 ; g_1) . 
 \]
 Then, 
 $v_j \longrightarrow v$ and the associated unit speed geodesics $\gamma_{1,j}:= \gamma_{z,v_j}$ in $(\R^3,g_1)$ approximate the segment $\gamma_1$. 
The segment $\gamma_1$ is then close to a shortest curve from $z$ to the boundary and therefore small perturbations of it hit the boundary transversally near the point $\gamma_1(l) = x$. 
 Consequently, we may assume that there is a sequence of positive real numbers $l_{j} \geq l$ such that $l_{j} \longrightarrow l$ and $ \gamma_{1,j} (l_{j} ) \in \Upsilon$. Let us denote by $\gamma_{2,j} $ the geodesics in $(\R^3,g_2)$ with the property $( \gamma_{2,j}(l_{j} ) ,  \dot\gamma_{2,j}(l_{j} ) ) = ( \gamma_{1,j}(l_{j} ) ,  \dot\gamma_{1,j}(l_{j} ) )$. 
These geodesics have unit speed and they approximate $\gamma_2$. 
 
 We now apply Proposition \ref{pinheiro} for $g=g_1$ using $(z,\theta,\beta + \beta'_j)$ and $(z,\tilde\theta,\beta + \beta'_j)$, $j=1,2,3,\dots$ as fixed parameters. 
 The point is that there must be a singularity in $u_\lambda |_{\lambda=1} $ that propagates along a bicharacteristic $\Sigma_j(r)$ of the form \eqref{optomet} emanating from $(z,0,-\epsilon \flat_{g_1} v_j,  \epsilon ) \in \bigcup_{j=1,2,3 } WF(S_j)$, where $\epsilon \in \{-1,1\}$. 
Moreover, since $\beta+\beta_j' > \inf \mathcal{I}_z > \beta_{ \partial W} (z)  $ by assumptions,  the propagating singularity for large $j$ can not hit the source and get annihilated before hitting $\Upsilon$  as otherwise it would contradict the definition of  $\beta_{ \partial W} (z)$. 
 Let us denote $WF(u) := \bigcup_{k=1,2,3}  WF (u_k)$. By putting together the proposition and the construction above, we conclude
 \[
 \Sigma_j (l_j) =  (\gamma_{1,j}(l_j) , l_j, -\epsilon\flat_{g_1} \dot\gamma_{1,j} ( l_j) , \epsilon ) \in WF ( u_\lambda ( \ \cdot \  ; z,\theta,\beta+\beta_j' ) )\big|_{\lambda=1} \cap WF ( u_\lambda ( \  \cdot \  ; z,\tilde\theta,\beta+\beta_j' ) ) \big|_{\lambda=1} 
 \]
 where $\epsilon \in \{-1,1\}$ and $u_\lambda$ is as in the assumptions. Denote $x_j := \gamma_{1,j}(l_j) = \gamma_{2,j}(l_j)$. 
%
%
 %
 %
 %
  The covector $ (x_j , l_j, -\epsilon \flat_{g_1} \dot\gamma_{1,j} ( l_j) , \epsilon )= (x_j , l_j, -\epsilon \flat_{g_2} \dot\gamma_{2,j} ( l_j) , \epsilon ) $, has a non-zero correspondent in 
 \[
 WF ( u_{\lambda=1} ( \ \cdot \  ; z,\theta,\beta+\beta_j' ) )\big|_{T(\Upsilon \times \R) } \cap WF ( u_{\lambda=1} ( \  \cdot \  ; z,\tilde\theta,\beta+\beta_j' ) ) \big|_{T(\Upsilon \times \R)} .
 \]
As the data for both metrics coincide, the set above equals 
 \begin{equation}\label{eroqs}
 \begin{split}
 WF ( u_{\lambda=2} ( \ \cdot \  ; z,\theta,\beta+\beta_j' ) )\big|_{T(\Upsilon \times \R) } \cap WF ( u_{\lambda=2} ( \  \cdot \  ; z,\tilde\theta,\beta+\beta_j' ) ) \big|_{T(\Upsilon \times \R)} 
  \end{split}
 \end{equation}
%
 Applying Proposition \ref{pinheiro}, this time with respect to the metric $g_2$, yields  that 
 \[
  \begin{split}
 &WF ( u_{\lambda=2} ( \  \cdot \  ; z,\theta,\beta+\beta_j' ) )  \setminus N^*K(z,\theta,\beta+\beta_j' ) \\
 & \subset \Lambda_{f,\epsilon} \circ  WF(S_{\lambda=2}(z,\theta,\beta+\beta_j' )  )    \\ 
 & \subset   \Lambda_{f,\epsilon} \circ \big( N^*K(z,\theta,\beta+\beta_j' ) \cap  T^* (B \times \R)   \big)  
    \end{split}
 \]
 where $\Lambda_{f,\epsilon} $ refers to the forward flowout canonical relation \eqref{jksks20kasfdbg} for $g=g_2$. 
Hence the projection of $ (x_j , l_j, -\epsilon \flat_{g_2} \dot\gamma_{2,j} ( l_j) , \epsilon )$ in \eqref{eroqs} lifts into two possible covectors in the wave front set $WF ( u_\lambda ( \  \cdot \  ; z,\theta,\beta+\beta_j' ) ) $, both lying in a curve of the form \eqref{optomet} emanating from $ \big( N^*K(z,\theta,\beta+\beta_j' ) \cap  T^* (B \times \R)   \big)$. 
These two covectors are  $ (x_j , l_j, -\epsilon \flat_{g_2} \dot\gamma_{2,j} ( l_j) , \epsilon )$ and $ (x_j , l_j, -\epsilon \flat_{g_2} \dot\gamma_{3,j} ( l_j) , \epsilon )$ where, analogously to $\gamma_3$, the $\gamma_{3,j}$ stands for the unit-speed geodesic in $(\R^3,g_2)$ through $x_j$ with velocity tangent to the mirror image of $\dot\gamma_{2,j}(l_j)$ over $T_{x_j} \partial W$ at that point. 
Thus at least one of these covectors lies  in a bicharacteristic \eqref{optomet} (for $g=g_2$) that emanates from $N^*K(z,\theta,\beta+\beta_j' ) \cap  T^*(B \times \R) $. In particular, one of the curves $r\mapsto (\gamma_{2,j} (r),r)$ and $r\mapsto (\gamma_{3,j} (r),r)$ intersects $K(z,\theta,\beta+\beta_j' ) \cap  (B \times \R)$. 
These curves approximate the geodesics  $r\mapsto (\gamma_{2} (r),r)$ and $r\mapsto (\gamma_{3} (r),r)$ arbitrarily well inside the bounded set $(B \times \R) \setminus \{(z,0)\}$ 
 so the intersection for large $j$  can take place only at the point $(z,0)$ due to the way the parameters $(\theta,\beta) \in M$ were fixed.  Thus, we deduce that $\gamma_{2,j}(0) = z$. In particular, each $\gamma_{2,j}$ is a geodesic from $z$ to $x_j=\gamma_{2,j}(l_j) \in \Upsilon$ in $(\R^3,g_2)$ with length $l_j$. Consequently, we have $ \text{dist}_{g_2}(x_j , z) \leq l_j $ which implies
 \[
\text{dist}_{g_2}(x , z) \leq l = \text{dist}_{g_1}( x,z)  
   \]
  at the limit $j \longrightarrow \infty$. Moreover, as $\gamma_{2,j}$ converges to $\gamma_2$ we conclude that $\gamma_2$ connects $z$ to $x$ in $(M_2,g_2)$ and has the distance above as length. 
  \end{proof}
  
  We can now apply the lemma above to deduce the following proposition: 
 
\begin{proposition}\label{234235rwerweght}
Assume that the conditions of Theorem \ref{maineee} hold. Then 
\[
\text{dist}_{g_1} (z,\partial W) =  \text{dist}_{g_2} (z,\partial W) 
\]
for every $ z \in U$. Moreover, a boundary point in $ \Upsilon $ is the nearest to $z$ in $(W,g_1)$ if and only if it is a nearest point in $(W,g_2)$ and for any such point $x_0 \in \Upsilon$ there exists a neighbourhood $\mathcal{Q}$ of $x_0$ in the space $\partial W$ such that 
\[
\text{dist}_{g_1} (z,x) =  \text{dist}_{g_2} (z,x), \quad \forall x \in \mathcal{Q}. 
\] 
\end{proposition}
\begin{proof}
 As a direct consequence of the lemma above we get that 
 \begin{equation}\label{rore}
 \text{dist}_{g_2}( z,\partial W) \leq \text{dist}_{g_2}( z, x_0) \leq \text{dist}_{g_1}( z,x_0) =\text{dist}_{g_1}( z,\partial W) .
 \end{equation}
By exchanging the roles of metrics and repeating the proof one derives the opposite inequality: 
\[
\text{dist}_{g_1} (z, \partial W) \leq \text{dist}_{g_2} (z, \partial W).
\]
Thus,
\[
\text{dist}_{g_1} (z, \partial W) =  \text{dist}_{g_2} (z, \partial W) . 
\]
Substituting the identity above in \eqref{rore} one deduces that 
\[
 \text{dist}_{g_2}( z,\partial W) =  \text{dist}_{g_2}( z,x_0),
\]
which implies that $x_0$ is a nearest boundary point also in $(W, g_2)$. 
Similarly one shows that a nearest boundary point in $(W, g_2)$ is nearest also in $(W, g_1)$. 
In view of Lemma \ref{fd5sol345} there is a neighbourhood $\mathcal{Q}$ of $x_0$  such that 
\[
 \text{dist}_{g_2}( z,x)  \leq  \text{dist}_{g_1}( z,x), \quad \forall x \in \mathcal{Q}. 
\]
 Again, we may swap the metrics and repeat the argument to obtain the reversed inequality for a sufficiently small neighbourhood $\mathcal{Q}$.
In conlusion, 
\[
 \text{dist}_{g_2}( z,x) =  \text{dist}_{g_1}( z,x), \quad \forall x \in \mathcal{Q}. 
\]
\end{proof}

We finish the section by deriving the main result from the proposition above:  

\begin{proof}[Proof of Theorem \ref{maineee}]
It suffices to show that $g_1|_{z} = g_2|_{z}$ at arbitrary $z\in U$. 
Let $x_0 \in \Upsilon$ be a nearest boundary point to $z$.  Let $\mathcal{Q}$ be a neighbourhood of $x_0$ in $\partial W$ as in the proposition above. 
As shown in \cite[Lemma 2.15]{MR1889089}, the distance function $ x \mapsto \text{dist}_{g_\lambda}(x,z)$ on a sufficiently small $\mathcal{Q}$ is smooth and the derivative $\phi_\lambda (x):= -d_w \text{dist}_{g_\lambda}(x,w) |_{w=z}$ defines 
a diffeomorphism from $\mathcal{Q}$ to an open set $\phi_\lambda \mathcal{Q}$  in $\{ \theta \in T_z\R^3 :  | \theta |_{g_\lambda}=1 \}$.\footnote{We unfortunately have the opposite roles for $z$ and $x$ compared to the notation in the reference.}
The derivative $\phi_\lambda(x)$ points towards $x$ along an optimal geodesic segment between $z$ and $x$. 
It follows from Proposition \ref{234235rwerweght} that $\phi_1 = \phi_2$ on $\mathcal{Q}$. 
Thus, for every $v$ in the open cone  $\Gamma_z  := \R_+ \phi_1 \mathcal{Q}= \R_+ \phi_2 \mathcal{Q} \subset T_z \R^3$ we have
\begin{equation}\label{opelw}
| v |_{g_1} = | v |_{g_2}, \quad \forall v \in  \Gamma_z. 
\end{equation}
Fix basis vectors $e_1,e_2,e_3 \in \Gamma_z $ for $T_z\R^3$ so close to each other that  $e_j + e_k \in \Gamma_z$ holds for every $j,k=1,2,3$.
The identity \eqref{opelw} then implies 
\[
g_1(e_j,e_k) = \frac{1}{2}  |e_j + e_k|_{g_1}^2 - \frac{1}{2} |e_j|_{g_1}^2-\frac{1}{2} |e_k|_{g_1}^2 = \frac{1}{2}  |e_j + e_k|_{g_2}^2 - \frac{1}{2} |e_j|_{g_2}^2-\frac{1}{2} |e_k|_{g_2}^2= g_2(e_j,e_k).
\]
In conclusion, for every $v= v^je_j $ and $w= w^k e_k$ in $T_z \R^3$ we have 
\[
g_1(v,w) = v^j w^k g_1(e_j,e_k)= v^j w^k g_2(e_j,e_k) = g_2(v,w),
\]
that is, $g_1|_z = g_2|_z$. 
\end{proof}

\appendix
\section{Supplemental Material for Section \ref{micromikko} }\label{appe1}


\begin{proof}[Proof of Lemma \ref{equdfkss}]
 Let us prove the case $\pm = +$. The proof for the other sign is similar. 
By assumptions, $P^k_j= p^k_j (x,D_x,D_t) $, where
\[
p^k_j (x,\xi,\omega)  = ( | \xi |_g^2 - \omega^2 ) \delta^k_j + r^k_j (x,\xi) , \quad j,k=1,2,3, 
\]
for some $r_j^k(x,\xi) \sim \sum_{m=-1}^\infty a_{j,-m}^{k}(x, \xi)$ (we may as well proceed with general $F \in M_3 (  \Psi^1_{cl} ( \R^3 ) ) $)  where $a_{j,-m}^k (x,\xi)$ is positively homogeneous of degree $-m$ with respect to $\xi$. 
Let us formally solve 
\[
   ( D_t \delta_j^l - W^l_j(x,D_x)  ) Z^k_l  -  P^k_j  \equiv 0  
\]
 by substituting an ansatz in the asymptotic form  
\begin{align}
&W^k_j = w^k_j(x,D_x), \quad  w^k_j (x,\xi)   \sim  \sum_{m=-1}^\infty w^k_{j , -m}(x,\xi) , \quad
 w^k_{j,1} :=    | \xi |_g \delta^k_j , \\
&Z^k_j = z^k_j(x,D_x,D_t), \quad   z^k_j (x, \xi , \omega)    \sim  \sum_{m=-1}^\infty z_{j,-m} ^k ( x, \xi, \omega )  ,  \quad 
z_{j,1}^k(x,\xi,\omega) := -  ( \omega + |\xi|_g ) \delta^k_j
\end{align}
where $w^k_{j , -m}$ and $z_{j,-m}$ are positively homogeneous of degree $-m$. 
We apply the standard formula 
\begin{equation}\label{gener}
Symbol (AB) \sim \sum_\alpha  \frac{1}{\alpha!}  (\partial_{\xi}^\alpha a ) ( D^\alpha b ) , \quad a:= Symbol (A), \quad b:= Symbol(B) ,
\end{equation}
(see e.g. \cite[Theorem 3.6]{Grigis-Sjostrand})
 put together terms that correspond to the same degree of homogeneity and 
assume that they sum to zero. 
Multiplying the highest order terms of the ansatz gives the correct principal part, $ ( | \xi |_g^2 - \omega^2 ) \delta^k_j $, whereas for 
the rest of the terms one deduces the following conditions:
\begin{align}
\omega z^k_{j,-m-1}  - \sum_{ s + h + |\alpha |= m } \frac{1}{\alpha !} ( \partial_\xi^\alpha w_{j,-h}^l ) ( D_x^\alpha z^k_{l, -s} ) = a_{j,-m}^{k}, \quad m=-1,0,1,2,3,\dots.
\end{align}
By extracting the terms corresponding to $s=m+1$ and $h= m+1$ we obtain  
\begin{align} 
&( \omega   -   | \xi |_g )   z_{j,-m-1}^k +w^k_{j , -m-1}   ( \omega + |\xi|_g ) -\sum_{  J_m    } \frac{1}{\alpha!} ( \partial_{\xi}^\alpha w_{j,-h}^l )( D^\alpha z^k_{l,-s } )  = a_{j,-m}^{k} ,  \label{saod}
  \\
  &J_m := \{ ( \alpha,h,s)  :   s+h = m - | \alpha |, \ | \alpha | \geq 1, \ s,h\geq -1  \} \subset \mathbb{N}^3 \times \{-1,\dots,m\}^2   \label{fikek}
\end{align}
for  $j,k = 1,2,3 $ and $m = -1,0, 1,2,3,\dots$ 
Fixing $\omega = | \xi|_g$ in the equation yields 
\[
w^k_{j , -m-1}(x,\xi) =\frac{1}{2|\xi|_g} \bigg( \sum_{  J_m    } \frac{1}{\alpha!} ( \partial_{\xi}^\alpha w_{j,-h}^l )( D^\alpha z^k_{l,-s } )   + a^{k}_{j,-m} (x, \xi)  \bigg)
\]
which solves $w^k_{j , -m-1}$ recursively from the terms of degree $-m,\dots,1$.  
Analogously, one derives $z_{j,-m-1}^k(x,\xi,\omega) = - w^k_{j , -m-1}(x,\xi) $ by setting $ \omega = -|\xi|_g$ in \eqref{saod} which also removes the terms with $\omega$. 
However, the elements obtained in this way have singularities at $\{ \xi = 0\}$. To get around this one redefines the terms by multiplying with the cut-off $\chi( \frac{ \omega   }{   2 \epsilon  |\xi| }  )$,  where $\chi \in C^\infty_c (\R)$ such that $\chi = 1$ in $(-1/2,1/2) $ and $0$ outside $(-1,1)$. 
 The cut-off does not change the elements in $\mathcal{V}$ and hence 
the claim follows for $h$ given by  $h^k_j \sim \sum_{m=0}^\infty w_{j-m}^k$. 
\end{proof}

\begin{proof}[Proof of Lemma \ref{wiikki}]
We begin by introducing an useful notation. 
A composition $V W$ of two operators $V\in \mps{l_1}$ and $W\in \mps{l_2}$  with symbol matrices $v \in  M_3( S^{l_1}( \R^4 \times \R^4) )$ and $w \in M_3 (S^{l_2}( \R^4 \times \R^4) )$ admits as a symbol the matrix $v \# w \in M_3 (S^{l_1+l_2} ( \R^4 \times \R^4) )$, given by 
 \[
 (v \# w)^k_j  (x,t,\xi,\omega) \sim \sum_{s=0}^\infty \sum_{|\alpha| =s } \frac{1}{\alpha!} (\partial^\alpha_\xi v^l_j (x,t,\xi,\omega))( D_x^\alpha w_l^k (x,t,\xi,\omega)), \quad j,k=1,2,3.
 \]
  Further, if $v$ and $w$ have classical asymptotic developments $v\sim \sum_{h=-l_1}^\infty v_{-h}$ and $w\sim \sum_{h=-l_2}^\infty w_{-h}$ for homogeneous $w_{-h},v_{-h} \in M_3(S^{-h}_{hom}( \R^4 \times \R^4)) $ of degree $-h$, we have the classical asymptotics
\[
v \# w \sim \sum_{h=-l_2} v \# w_{-h} := \sum_{s=-l_1-l_2}^\infty \sum_{h=-l_2}^{l_1+s} v_{h-s} \# w_{-h} 
\]
The proof of these formulas reduce to the scalar case (cf. \cite[Theorem 3.6]{Grigis-Sjostrand} or \cite[Theorem 2.5.2]{Duistermaat}) by considering each entry individually. 

Let us now continue to the actual proof of the lemma. 
Set $a := ( a^k_j)_{j,k=1,2,3} \in M_3 ( S^{m}_{cl} ( \R^4 \times \R^4))$ where the entry $a^k_j \in S^{m}_{cl} ( \R^4 \times \R^4)$ stands for the symbol of $L^k_j$.
Each entry admits an asymptotic expansion $a_j^k  \sim \sum_{h=-m}^\infty  a_{j,-h}^k$, where $a_{j,-h}^k(x,t,\xi,\omega)$ is positively homogeneous of degree $-h$ in $\xi,\omega$ and we define $a_{-h} = (a_{j,-h}^k)_{j,k=1,2,3} \in M_3( S^{-h} ( \R^4 \times \R^4  ) )$. Without loss of generality we may fix $\epsilon,\delta>0$ such that  $a_m$ is invertible in the conic set 
 \begin{equation}\label{rorfe2d}
 \Sigma_{\epsilon,\delta} := \left\{ (x,t,\xi,\omega) \in \R^4 \times( \R^4 \setminus 0 ): |(x,t)-(x_0,t_0)|\leq \epsilon, \ \left| \frac {(\xi,\omega)}{ | (\xi,\omega)| } - \frac{ ( \xi_0, \omega_0) }{|( \xi_0,\omega_0)|} \right| \leq \delta \right\}  .
 \end{equation}
As the matrix $a_m$ is positively homogeneous of degree $m$ in $\xi,\omega$, the inverse matrix $a^{-1}_m$ on $ \Sigma_{\epsilon,\delta}$ is positively homogeneous of degree $-m$. By proceeding as in the proof of \cite[Theorem 4.1]{Grigis-Sjostrand}  we construct a suitable matrix $b= (b^k_j)_{j,k=1,2,3}$ where each $b^k_j$ defines the symbol of $R^k_j$, $j=1,2,3$. 
$b \sim  \sum_{h=m}^\infty b_{-h}$. 
Let $\chi \in C^\infty(\R)$ be a cut-off function such that $\chi|_{(0,\delta/2)}=1$ and $\text{supp}(\chi) \subset (0,\delta)$.    Set $\mathcal{V}:= \Sigma_{\epsilon/2,\delta/2}$. 
We define the  term $b_{-m}$  by 
\[
b_{-m} (x,t,\xi,\omega) := \begin{cases}  \tilde\chi(x,t,\xi,\omega)  (a_m)^{-1} (x,t,\xi,\omega)  , & (x,t,\xi,\omega) \in \Sigma_{\epsilon,\delta} \\
0 , & (x,t,\xi,\omega) \in T^*\R^4 \setminus  \Sigma_{\epsilon,\delta} ,
\end{cases}
\]
where 
\begin{align}
\tilde\chi(x,t,\xi,\omega) :=  \chi(  |(x,t)-(x_0,t_0)| )  \chi \Big( \left| \frac {(\xi,\omega)}{ | (\xi,\omega)| } - \frac{ ( \xi_0, \omega_0) }{|( \xi_0,\omega_0)|} \right| \Big). 
\end{align} 
Then $b_{-m}$ is positively homogeneous of degree $-m$ and equals $(a_m)^{-1}$ on $\mathcal{V}$. 
The leading terms in the asymptotic expansions $a\#b_{-m} \sim \sum_{h=-m}^\infty a_{-h} \# b_{-m}$ and  $b_{-m}\#a \sim \sum_{h=-m}^\infty b_{-m} \# a_{-h}$ are the compositions $a_m b_{-m}$ and $b_{-m}a_m$ which in $\mathcal{V}$ reduce to  the identity matrix by definition.
Thus, we can find $f,g \in M_3( S_{cl}^{-1}(\R^4 \times \R^4))$ such that for $(x,t,\xi,\omega) \in \mathcal{V}$,
\begin{align}
(I-f)(x,t,\xi,\omega) &=  a\#b_{-m} (x,t,\xi,\omega),  \\
(I- g)(x,t,\xi,\omega) &= b_{-m}\#a (x,t,\xi,\omega) ,
 \end{align}
 where $I$ stands for the identity matrix. 
 Define
\begin{align}
&b_f \sim  b_{-m}\# (I + f +f\#f +f\#f\#f +\dots),  \\
&b_g \sim   (I + g +g\#g +g\#g\#g +\dots) \# b_{-m}.
\end{align}
In $  \mathcal{V} $ we deduce $a \# b_f \equiv I \equiv b_g \# a  $ modulo $  M_3( S^{-\infty}( \R^4 \times \R^4) ) $ by substituting the identities above. 
Consequently,
\[
WF(I - LR_f) \cap   \mathcal{V} = \emptyset, \quad WF(I - R_g L) \cap \mathcal{V} = \emptyset,
\]
where 
\[
\begin{split}
(R_f)^k_j  :=  (b_f)^k_j(x,t,D_x,D_t) , \quad 
(R_g)^k_j  := (b_g)^k_j (x,t,D_x,D_t). 
\end{split}
\]
Further,
\[
b_g \equiv I \# b_g \equiv  ( b_f \# a ) \# b_g  \equiv   b_f \# (a  \# b_g ) \equiv b_f \# I \equiv b_f \quad \text{in} \quad   \mathcal{V}
\]
which implies $R_f \equiv R_g$. 
In conclusion, we may define the terms with lower degree of homogeneity by
\[
\begin{split}
 b_{-m-h} := b_{-m-h+1} \# f, \quad h=1,2,3,\dots,
\end{split}
\]
(i.e. $R:= R_f=R_g$) to obtain the result.
 \end{proof}
 
 \begin{proof}[Proof of Lemma \ref{ds89askqeiorw}]
Let us show that the bicharacteristics are of the form \eqref{optomet}. Denote $H := H_{p^\pm}$, $p^\pm(x,t,\xi,\omega) :=   \omega \mp |\xi|_g$.  By definition,
\[
H = \frac{\partial }{\partial t} \mp  \frac{ g^{jk} \xi_j } { | \xi|_g} \frac{\partial }{ \partial x^k}  \pm   \frac{ \xi_l \xi_j } { 2 | \xi|_g} \frac{\partial g^{jl} }{\partial x^k}   \frac{\partial }{ \partial \xi_k}  .
\]
Taking a derivative of $g^{jl}g_{lm}  = \delta^j_m$ yields $g_{lm} \partial_k g^{jl} +g^{jl} \partial_k g_{lm}  = 0$. That is, $ \partial_k g^{jl} =- g^{lm }g^{jh} \partial_k g_{hm} $. Hence, 
\[
H = \frac{\partial }{\partial t}  \mp  \frac{ g^{jk} \xi_j } { | \xi|_g} \frac{\partial }{ \partial x^k}  \mp    \frac{ (\xi_l g^{lm}) (\xi_j g^{jh}) } { 2 | \xi|_g} \frac{\partial g_{mh} }{\partial x^k}\frac{\partial }{ \partial \xi_k}  .
\]
Thus, a bicharacteristic curve satisfies
\begin{align}
&\dot{X}^k (r) =  \mp  \frac{ g^{jk}(X(r)) \Xi_j(r) } { | \Xi(r) |_g} , \label{ewiwe1}\\
&\dot{T} (r) = 1,\label{ewiwe2} \\
 & \dot\Xi_k(r) =  \mp \frac{1}{2 | \Xi(r) | }   \Xi_l g^{lm}(X(r)) \Xi_j g^{jh} (X(r))  \frac{\partial g_{mh} }{\partial x^k} \Big|_{X(r)} \label{ewiwe3}, \\
& \dot\Omega = 0 \label{ewiwe4}. 
\end{align}
An initial value at $r=0$ fixes the solution uniquely. 
The equations \eqref{ewiwe1}, \eqref{ewiwe2} and \eqref{ewiwe4} are solved by 
\[
X(r) = \gamma_{x,v}\left(\frac{r}{|v|_g}\right), \quad  \Xi (r) = \mp \flat_g \dot \gamma_{x,v}\left(\frac{r}{|v|_g}\right), \quad  T(r) = t +   r, \quad \Omega = \pm | v|_g ,
\]
where the initial value $\Gamma(0) = (x,t,\xi,  \pm |\xi|_g)$ is obtained by setting $v := \mp \sharp_g \xi$. 
Let us show that this solves also \eqref{ewiwe3}. It is suffices to show that
\begin{equation}\label{tuleasssw}
  \partial_r \left(    \flat_g \dot \gamma_{x,v} \left(\frac{r}{|v|_g}\right)  \right)_k=   \frac{1}{2 | v |_g }   \dot\gamma^m_{x,v}  \left(\frac{r}{|v|_g}\right)  \dot\gamma_{x,v}^h  \left(\frac{r}{|v|_g}\right)  \partial_k g_{mh}  \big|_{\gamma_{x,v}(r/|v|_g)}.
\end{equation}
We compute
\[
\begin{split}
&\partial_r \left(   \flat_g \dot \gamma_{x,v} \left(\frac{r}{|v|_g}\right)  \right) =\partial_r \left(   g_{lk}\left( \gamma_{x,v} \left( \frac{r}{|v|} \right) \right)  \dot \gamma_{x,v}^l \left(\frac{r}{|v|_g}\right)  \right)  \\
&=  \frac{1}{|v|_g} \dot\gamma^j_{x,v}\left(\frac{r}{|v|_g}\right)   \dot \gamma_{x,v}^l \left(\frac{r}{|v|_g}\right)  \partial_j g_{lk}  \big|_{\gamma_{x,v}(r/|v|_g)}+\frac{1}{|v|_g}  \ddot \gamma_{x,v}^h \left(\frac{r}{|v|_g}\right)   g_{hk}\left( \gamma_{x,v}\left(\frac{r}{|v|_g}\right) \right) .
\end{split}
\]
As $\gamma_{x,v}$ is a geodesic, one derives
\[
  g_{hk} (\gamma)  \ddot \gamma_{x,v}^h =  -\dot\gamma_{x,v}^j\dot\gamma_{x,v}^l g_{hk}(\gamma) \Gamma_{jl}^h(\gamma)  =- \dot\gamma_{x,v}^j\dot\gamma_{x,v}^l  \partial_j g_{lk}|_\gamma + \frac{1}{2} \partial_k g_{jl} |_\gamma \dot\gamma_{x,v}^j\dot\gamma_{x,v}^l 
\]
where $\Gamma_{jl}^h$, $j,l,h=1,2,3$, are the Christoffel symbols of $g$. 
Substituting this in the previous equation yields \eqref{tuleasssw}. 
\end{proof}

\begin{proof}[Proof of Lemma \ref{lahjboelihjs125}] 
Since the Hamiltonian vector field corresponding to $D_s$ is simply $H_{\rho} = \partial_s$, it follows immediately that the bicharacteristics of $D_s$ (i.e. the characteristic integral curves of the Hamiltonian vector field) equal  $r \mapsto (y,s + r,\eta,0) $, for $(y,s,\eta) \in \R^4 \times \R^3$.
By Lemma \ref{ds89askqeiorw} it suffices to show that the curve $ F : r \mapsto \mathcal{H} (y,s+r,\eta,0)$ is an integral curve of the Hamiltonian $H_{p^\pm}$. For $f \in C^\infty ( T^*\R^4)$ we compute (See [Duistermaat, Theorem 3.5.2.])
\[
\begin{split}
H_{p^\pm}( F(r) )  f  &=  \{ p^\pm,f\} \circ \mathcal{H}  ( y, s+ r, \eta,  0 ) 
= \{ p^\pm \circ \mathcal{H} ,f \circ \mathcal{H} \}    ( y, s + r, \eta,  0 ) \\
&= \{ \rho , f\circ \mathcal{H}\}  ( y, s + r, \eta,  0 )  
=  \dot{F}(r)  f.    \\
\end{split}
\]
This proves the claim. 
\end{proof}

\begin{proof}[Derivation of \eqref{meniscus_medialis}]
Recall that $\mathcal{X} \mathcal{Y} \equiv I $ in a small conic neighbourhood $ \mathcal{V}$ of $(x_0,t_0,\xi_0,\omega_0)$ and that the operator wave-fronts of $\mathcal{X}$ and $ \mathcal{Y}$ are the graphs $G_{\mathcal{H}}$ and $G_{\mathcal{H}^{-1}}$, respectively.
Hence, 
\[
\begin{split}
\mathcal{V} \cap   \bigcup_{j=1,2,3 }  WF(f_j )  = \mathcal{V} \cap   \bigcup_{j=1,2,3 }   WF( \mathcal{X}^l_j \mathcal{Y}_l^k f_k ) \\
\subset \mathcal{V} \cap   \bigcup_{j=1,2,3 }   G_{\mathcal{H}} \circ  WF(  \mathcal{Y}_j^k f_k )  = \mathcal{V} \cap   \bigcup_{j=1,2,3 }   \mathcal{H} WF(  \mathcal{Y}_j^k f_k ) 
\\
 \subset \mathcal{V} \cap   \bigcup_{j=1,2,3 }   \mathcal{H} G_{\mathcal{H}^{-1}} \circ  WF(   f_j )  = \mathcal{V} \cap   \bigcup_{j=1,2,3 }     WF(   f_j ) ,
 \end{split}
\]
\[
\Rightarrow   \mathcal{V} \cap   \bigcup_{j=1,2,3 }     WF(   f_j ) =  \mathcal{V} \cap   \bigcup_{j=1,2,3 }   \mathcal{H} WF(  \mathcal{Y}_j^k f_k )  .
\]
Define a conic neighbourhood $\mathcal{W} \subset T^*\R^4$ by $\mathcal{W} = \mathcal{H}^{-1} \mathcal{V}$.  Then,
\[
\begin{split}
\Rightarrow    \mathcal{W} \cap  \mathcal{H}^{-1}\bigcup_{j=1,2,3 }     WF(   f_j ) 
=  \mathcal{H}^{-1}  \bigg(   \mathcal{V} \cap   \bigcup_{j=1,2,3 }     WF(   f_j )\bigg)
\\
=   \mathcal{H}^{-1} \bigg(  \mathcal{V} \cap   \bigcup_{j=1,2,3 }   \mathcal{H} WF(  \mathcal{Y}_j^k f_k )  \bigg)  
=     \mathcal{H}^{-1} \mathcal{V} \cap   \bigcup_{j=1,2,3 }  \mathcal{H}^{-1}  \mathcal{H} WF(  \mathcal{Y}_j^k f_k ) 
\\ 
=    \mathcal{W} \cap   \bigcup_{j=1,2,3 }   WF(  \mathcal{Y}_j^k f_k )  .
\end{split}
\]
Hence, 
\[
  \mathcal{W} \cap  \mathcal{H}^{-1}\bigcup_{j=1,2,3 }     WF(   f_j ) =    \mathcal{W} \cap   \bigcup_{j=1,2,3 }   WF(  \mathcal{Y}_j^k f_k )  .
\]

\end{proof}

\bibliographystyle{alpha}
\bibliography{bibliography_cherenkov}

\end{document}